\documentclass[english]{amsart}
\usepackage[T1]{fontenc}
\usepackage[latin9]{inputenc}
\usepackage{amsthm}
\usepackage{amstext}
\usepackage{graphicx}
\usepackage{amssymb}
\usepackage{esint}
\makeatletter
\numberwithin{equation}{section} 
\numberwithin{figure}{section} 
\theoremstyle{plain}
\newtheorem{theorem}{Theorem}
  \theoremstyle{plain}
  \newtheorem{proposition}[theorem]{Proposition}
  \theoremstyle{definition}
  \newtheorem{definition}[theorem]{Definition}
  \theoremstyle{plain}
  \newtheorem{lemma}[theorem]{Lemma}
  \theoremstyle{plain}
  
  \theoremstyle{remark}
  \newtheorem*{rem*}{Remark}

\usepackage{subfigure}\usepackage{epsfig}\usepackage{eepic}\usepackage{latexsym}
 \usepackage{amsmath}

\def \ve{\varepsilon}
\def \O{\Omega}

\def\XXint#1#2#3{{\setbox0=\hbox{$#1{#2#3}{\int}$}
\vcenter{\hbox{$#2#3$}}\kern-.87\wd0}}

\def\XXiint#1#2#3{{\setbox0=\hbox{$#1{#2#3}{\int}$}
\vcenter{\hbox{$#2#3$}}\kern-1.05\wd0}}

\def\XXintt#1#2#3{{\setbox0=\hbox{$#1{#2#3}{\int}$}
\vcenter{\hbox{$#2#3$}}\kern-.72\wd0}}

\def\Xinttt#1{\mathchoice
{\XXinttt\displaystyle\textstyle{#1}}%
{\XXinttt\textstyle\scriptstyle{#1}}%
{\XXinttt\scriptstyle\scriptscriptstyle{#1}}%
{\XXinttt\scriptscriptstyle\scriptscriptstyle{#1}}%
\!\int}
\def\XXinttt#1#2#3{{\setbox0=\hbox{$#1{#2#3}{\int}$}
\vcenter{\hbox{$#2#3$}}\kern-.52\wd0}}

\def\XXintttr#1#2#3{{\setbox0=\hbox{$#1{#2#3}{\int}$}
\vcenter{\hbox{$#2#3$}}\kern-.6\wd0}}

\def\XXintttt#1#2#3{{\setbox0=\hbox{$#1{#2#3}{\int}$}
\vcenter{\hbox{$#2#3$}}\kern-.78\wd0}}

\def\sqr#1#2{{\vcenter{\vbox{\hrule height.#2pt\hbox{\vrule width.#2pt height#1pt \kern#1pt\vrule width.#2pt}\hrule height.#2pt}}}}
\def\qed{\hfill$\sqr45$\bigskip}

\def\ddashinttt{\Xinttt-}

\usepackage{a4wide}

\makeatother

\usepackage{babel}

\title[Locally-periodic two-scale convergence]{Two-scale convergence for locally-periodic microstructures and homogenization of  plywood structures.}

\author{Mariya Ptashnyk}

\begin{document}

\maketitle

\begin{abstract}
The introduced  notion of locally-periodic  two-scale convergence  allows to average a wider range of microstructures, compared to the periodic one.  The compactness theorem for the locally-periodic  two-scale convergence  and the characterisation of  the limit for  a sequence bounded in $H^1(\Omega)$ are proven.  The underlying analysis  comprises the  approximation of   functions,  which periodicity  with respect to the fast variable depends on the slow  variable, by locally-periodic functions, periodic in subdomains smaller than  the considered domain, but larger than the size of  microscopic structures. The  developed theory  is applied to derive  macroscopic equations  for a  linear elasticity problem defined in  domains with plywood structures. 
\end{abstract}

%
%
\pagestyle{myheadings}
\thispagestyle{plain}

\section{Introduction}
Many natural and man made  composite materials comprise  non-periodic microscopic structures, for example   fibrous microstructure with varied orientation of fibres in  heart muscles, \cite{Peskin},  in exoskeletons, \cite{Raabe2009_2},  in  polymer membranes and   industrial filters, \cite{Schweers}, or space-dependent perforations in concrete, \cite{Roy}. An interesting and important  special case of  non-periodic microstructures is the so called  locally-periodic microstructures,  where   spatial changes of the microstructure are observed on a  scale smaller than the size of the considered domain but  larger  than the characteristic  size of the microstructure.  The  distribution of  microstructures  in   locally-periodic materials  is known a priori,  in contrast to the stochastic description of the medium considered in stochastic homogenization, \cite{Andro}.  
 
There are  few mathematical results on  homogenization in locally-periodic and  fibrous media.  The homogenization of a heat-conductivity problem defined in locally-periodic and  non-periodic  domains consisting of spherical  balls   was studied in \cite{Briane3} using the Murat-Tartar $H-$convergence method, defined in \cite{Murat}.  The locally-periodic and non-periodic distribution of  balls is given by a $C^2-$diffeomorphism  $\theta$, transforming the centres of the balls. For the derivation of macroscopic equations for  the  problem posed in the non-periodic domain, where the changes of the microstructure  are given on the scale of the considered microstructure,  a locally-periodic  approximation was considered.  Estimates for the numerical approximation of this problem were derived in \cite{Shkoller}.  
The notion of $\theta-2$ convergence, motivated by  the homogenization in a domain  with a microstructure of non-periodically  distributed  spherical balls, was introduced in \cite{Alexandre}.  The Young measure was used in  \cite{Mascarenhas} to extend  the concept of periodic two-scale convergence, presented in \cite{Allaire}, and to define the so-called {\it scale convergence}. The definition of  scale convergence was mainly  motivated by the derivation of  the $\Gamma$-limit for a sequence of nonlinear energy functionals involving non-periodic oscillations. It has been shown that the two-scale and  multi-scale convergences are particular cases of the scale convergence.  In  \cite{Mascarenhas}, as an example of non-periodic oscillations, the domain with  perforations given by  the transformation of centre of balls was considered. 
Macroscopic models for  non-periodic fibrous materials were presented in \cite{Briane2} and derived in \cite{Briane1}.    The non-periodic fibrous material is characterised by gradually rotated  planes of parallel aligned fibres.  By applying the $H-$convergence method for a locally periodic approximation  of the non-periodic microstructure the effective homogenized matrix was derived.
The asymptotic expansion method was used  in \cite{Mikelic} to derive macroscopic equations for a filtration problem
 through a locally-periodic fibrous medium.   A  formal asymptotic expansion was also applied to derive a macroscopic model  for a Poisson equation, \cite{Chechkin1, Chechkin}, and  for convection-diffusion equations,  \cite{AdrianTycho1, AdrianTycho2}, defined  in  domains with  locally-periodic perforations, i.e. domains consisting of  periodic cells with  smoothly changing perforations. 
Two-scale convergence, defined for periodic test functions, was applied in \cite{Mascarenhas1, Mascarenhas3} to homogenize  warping,
 torsion and Neumann problems  in a two-dimensional domain with a smoothly changing perforation.  
Optimization of the corresponding homogenized problems was considered in \cite{Mascarenhas2, Polisevski}.
Locally-periodic perforation is related to the locally-periodic microstructure, considered in this work, so that  the  changes in the perforation are given on the $\ve$-level  and can be approximated by the locally-periodic microstructure, which is periodic in each subdomain of size $\sim\ve^{dr}$,  where $0<r<1$,  $\ve$ is the size of the periodic cell,  and $d$ is the dimension of considered domain.

 Two-scale convergence is a special type of the  convergence in $L^p$-spaces. It was introduced by Nguetseng, \cite{Nguetseng},  further developed in  \cite{Allaire, Lukkassen} and  is widely used for the homogenization of partial differential equations with  periodically oscillating coefficients or  problems posed in   media with periodic microstructures or with locally-periodic perforations (named also as quasi-periodic perforations).  Admissible test functions  used in the definition of two-scale convergence, are functions dependent on two variables: the fast microscopic and slow macroscopic variable, and periodic with respect to the fast variable. The two-scale convergence conserves the information regarding oscillations of the considered function sequence and overcomes difficulties resulting from  weak convergence of fast oscillating periodic functions.

In this article we generalise the notion  of the  two-scale convergence to locally-periodic situations, see Definition~\ref{def_two-scale}.  The considered test functions are locally-periodic  approximations of the corresponding functions with  the space-dependent periodicity  with respect to the fast variable  being  dependent on the slow variable. This generalised notion of the two-scale convergence provides easier and  more general  techniques for homogenization of partial differential equations with locally-periodic coefficients or considered in domains with  locally-periodic microstructures. The central result of the work is the compactness Theorem~\ref{L2_lp_covregence} for the locally periodic two-scale convergence and the characterisation of  the locally-periodic two-scale limit for  a sequence bounded in $H^1(\Omega)$,  Theorem~\ref{compactness_theorem}.
The proofs indicates  that  local periodicity of the considered microstructure is essential for the convergence of spatial derivatives.  Due to the definition of a locally-periodic domain  we have that the size of the subdomains with periodic microstructure is of order $\ve^{dr}$, where $0<r<1$ and $\ve$ is the size of the microstructure. Thus, the gradient of the smooth approximations of characteristic functions of  subdomains with periodic microstructures multiplied by $\ve$ is of order $\ve^{1-\rho}$, with $r<\rho<1$, and converges to zero as $\ve \to 0$.  This fact allows us to approximate the locally-periodic test function by differentiable functions and show the convergence of  spatial derivatives. 
 
In  Section~\ref{Application} we apply the locally-periodic two-scale convergence to derive  macroscopic mechanical properties of  biocomposites, comprising  non-periodic microstructures. As  an example of such a microstructure  we consider  the  plywood structure of the exoskeleton of a lobster, \cite{Raabe2009_2}.  Mechanical properties of the biomaterial are modelled by equations of linear elasticity, where the microscopic geometry and elastic properties of different components are reflected in the stiffness matrix of  the microscopic   equations. The fully non-periodic microstructure is approximated by a locally-periodic  domain, provided the transformation matrix, describing the microstructure, is twice continuously differentiable.  Our calculations for the fully non-periodic  situation were inspired by \cite{Briane2}.   Note that the techniques developed in this article can be applied  to derive macroscopic equations for a wide class of partial differential equations,  and are not restricted to the problem of linear elasticity. The introduced convergence is  also applicable to more general non-periodic transformations  as transformation of centres of spherical balls.
   
\section{Description of plywood structure}\label{ModelD} 
A major  challenge in   material science is the design of stable but light materials. Many biomaterials  feature excellent mechanical properties, such as   strength  or stiffness, regarding  their low density. For example, the  strength of bone is similar to that of steel, but it is three times lighter and ten times more flexible, \cite{Gao2004}.  Recent research suggests that this phenomenon is primarily   a consequence of the hierarchical structure of biomaterials over several length scales, \cite{Raabe2009}.  A better understanding of  the influence of microstructure on mechanical properties of  biomaterials is not only a theoretical challenge by itself but may also help to improve the design and production of synthetic materials
  
Here we consider  the exoskeleton of  a lobster as an example of such biomaterials. The exoskeleton is a hierarchical composite consisting of chitin-protein fibres, various proteins, mineral nanoparticles and water.  A prototypical pattern  found in the exoskeleton is the so-called  twisted plywood structure, 
given as the superposition  of planes of parallel aligned chitin-protein fibres, gradually rotated with  rotation angle $\gamma$,  \cite{Raabe2009_2}.

We would like to study  elastic properties of a exoskeleton. In the  formulation of the microscopic model defined on the scale of a single fibre,  we shall distinguish between mechanical properties of fibres and  inter-fibrous space. We assume that the fibres are cylinders of radius $\ve a$, perpendicular to the $x_3$-axis, whereas   $0<a <1/2$  and  $\ve>0$. It has been observed that   different parts of the exoskeleton  comprise different  rotation densities, \cite{Raabe2009_2}. 
Thus,  we shall distinguish between the locally-periodic  plywood structure, where the hight of layers of fibres aligned in the same direction is  of order  $\ve^r$ with $0<r<1$, and  the non-periodic microstructure, where each layer of fibres is rotated by a different angle, i.e. $r=1$, see Fig.~\ref{fig12} and Fig.~\ref{fig_non_per_plywood} in Section~\ref{Application}.  In the derivation of effective macroscopic equations the non-periodic situation shall be approximated by the locally-periodic microstructure. The case  $r=0$ would imply the periodic microstructure  and will not be considered here.

\begin{figure}
\begin{center}
\includegraphics[width=0.85\textwidth]{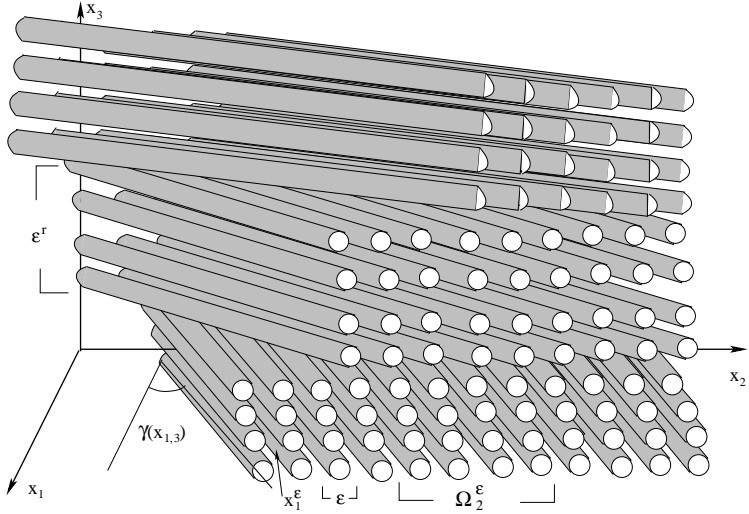}
\caption{\small{ Locally-periodic plywood structure. }} \label{fig12}
\end{center}
\end{figure}

In order to define the characteristic function of the domain occupied by fibres for locally-periodic plywood structure we divide  $\mathbb R^3$ into perpendicular to the $x_3$-axis layers $L_k^\ve=\mathbb R^2\times((k-1)\ve^r, k\ve^r)$  of  height  $\ve^r$, where $k \in \mathbb Z$ and  $0<r<1$, and in each $L_k^\ve$ choose  an arbitrary  point  $x_k^\ve\in L_k^\ve$.  In   $\mathbb R\times \hat Y$, with $\hat Y=[0,1]^2$,  we consider  the characteristic function  $\tilde\eta$  of a  cylinder of radius $a$ 
\begin{equation}\label{char}
\tilde \eta(y)= \begin{cases}
          1, \quad  |\hat y-1/2| \leq a ,\\
0, \quad |\hat y-1/2| > a,
         \end{cases}
\end{equation}
where $\hat y = (y_2, y_3)$,  and  extended  $\hat Y$-periodically to the whole of $\mathbb R^3$. 

For a Lipschitz bounded domain $\Omega$ we define $\Omega_f^\ve\subset\Omega$ as  the subdomain occupied by fibres. Then the characteristic function of $\Omega^\ve_f$ is given by 
\begin{equation}\label{charac_local_per}
\chi_{\Omega_f^\ve}(x)= \chi_\Omega(x)\sum\limits_{k\in \mathbb Z } \eta_\ve (R_{x_k^\ve} \, x) 
\chi_{L_k^\ve}(x),
\end{equation}
 with  $\eta_\varepsilon(x)= \eta(x/\ve)$ and $R_{x_k^\ve}= R(\gamma(x_{k,3}^\ve))$,   where $\gamma \in C^2(\mathbb R)$  is a given function,   $0\leq \gamma(x) \leq \pi$ for  $x\in \mathbb R$,  and  $R(\alpha)$  is the inverse of  the rotation  matrix   around the  $x_3$-axis  with rotation angle $\alpha$ with the $x_1$-axis 
\[
R(\alpha)\hspace{-0.05 cm}=\hspace{-0.1 cm} \left(\hspace{-0.1 cm}
\begin{array}{ccc}
 \phantom{-}\cos(\alpha) &\sin(\alpha)& 0\\
-\sin(\alpha) &\cos(\alpha)& 0\\
 0 & 0 & 1
\end{array}
\right).
\]
 The assumed regularity of $\gamma$ will be essential for the  homogenization analysis in Section~\ref{Application}. By $\chi_A$   we denote the characteristic function of a domain $A$.
 
In the context of  the  theory of linear elasticity,  which is widely used  in the continuous mechanics of solid materials, \cite{Lemaitre},  we describe  the elastic properties of a material with  the plywood structure    
\begin{equation}\label{main}
\begin{cases}
-\text{div}  \left(E^\ve(x)\,e(u^\varepsilon)\right)=G  \quad & \text{ in } \Omega, \\
 u^\varepsilon= g \qquad & \text{ on } \partial\Omega,
\end{cases}
\end{equation} 
where
$  e _{ij}(u^\varepsilon)= \frac 12\left(\frac{\partial u^\varepsilon_i}{\partial x_j}
 + \frac{\partial u^\varepsilon_j}{\partial x_i}\right)
$ and  the elasticity tensor   $E^\ve$ is given  by 
\begin{equation}\label{stif_loc_period}
E^\ve(x)= E_1\chi_{\Omega_f^\ve}(x) +  E_2\big(1-\chi_{\Omega_f^\ve}(x)).
\end{equation}
Here  $E_1$ is the stiffness tensor of the chitin-protein fibres and  $E_2$ is the stiffness tensor of the inter-fibre space. The boundary condition and the right hand side  describe the external forces applied to the material.

The considered microstructure is locally-periodic, i.e. it is periodic in each layer $L_k^\ve$, where  the height of the layer  $\ve^r$  is larger as the radius of a single fibre, but  is small compared to the size of the considered domain $\Omega$, provide  $\ve$ is sufficiently small. 

To allow more general  locally-periodic oscillations and microstructures, we shall consider the partition of $\Omega$ into cubes of side $\ve^r$ in the definition of the locally-periodic two-scale convergence. In the    locally-periodic plywood-structure   the rotation angle is constant in each layer $L_k^\ve$ and the characteristic function of  $\Omega^\ve_f$ can  also be defined considering the additional  division of  $L_k^\ve$ into  cubes  of  side $\ve^r$. Thus,   the introduced below  locally-periodic two-scale convergence is applicable to fibrous media. Although  the division of $\Omega$ into layers $L_k^\ve$ is  sufficient to define the microstructure of  the locally periodic plywood-structure, for the homogenization of  a non-periodic plywood-structure the partition covering of $\Omega$ by cubes  will be essential, see Section~\ref{Application}.
 
\section{Two-scale convergence for locally-periodic microstructures}\label{TwoScaleConv}
First we introduce the space-dependent periodicity and the corresponding function spaces.
Let $\Omega\subset \mathbb R^d$ be a bounded Lipschitz domain. 
For each $x \in \mathbb R^d$ we  consider a transformation matrix   $D(x)\in \mathbb R^{d\times d}$ and its inverse   $D^{-1}(x)$,  such that 
$D, D^{-1} \in \text{Lip}(\mathbb R^d; \mathbb R^{d\times d})$ and  $0<D_1\leq |\det D(x)| \leq D_2<\infty$ for all $x\in \overline \Omega$.
 For convenience,  we shall use the notations $D_x:=D(x)$ and $D_x^{-1}:=D^{-1}(x)$. By   $Y=[0,1]^d$ we denote the so-called 'unit cell' and consider  the  continuous family of  parallelepipeds   $Y_x=D_xY$ on $\overline\Omega$. 

We shall consider the space   $C(\overline\Omega; C_{\text{per}}(Y_x))$ given in a standard way, i.e. for any $\tilde \psi \in C(\overline\Omega; C_{\text{per}}(Y))$ the relation   $\psi(x,y)= \tilde \psi(x, D_x^{-1}y)$  with $x\in \Omega$ and $y \in Y_x$ yields   $\psi \in C(\overline\Omega; C_{\text{per}}(Y_x))$. In the same way  the spaces $L^p(\Omega; C_{\text{per}}(Y_x))$, $L^p(\Omega; L^q_{\text{per}}(Y_x))$  and  $C(\overline\Omega; L^q_{\text{per}}(Y_x))$,  for $1\leq p\leq\infty$, $1\leq q <\infty$,  are given.
 
The separability of  $C_{\text{per}}(Y_x)$  for each $x \in \Omega$ and the Weierstrass approximation for continuous functions  
$u \in (\Omega \to C_{\text{per}}(Y_x))$ imply the separability of $C(\overline \Omega; C_{\text{per}}(Y_x))$. 
 For  the norm  $ \| \psi\|_{C(\overline \Omega; C_{\text{per}}(Y_x))}:=
\sup\limits_{x\in \overline \Omega} \sup\limits_{y\in  Y_x} |\psi(x,y)| $
we have   the  relations
\[
\| \psi\|_{C(\overline \Omega; C_{\text{per}}(Y_x))}= 
\sup\limits_{x\in \overline\Omega} \sup\limits_{y \in  Y_x} |\tilde \psi(x,D_x^{-1}y)|=
\sup\limits_{x\in \overline\Omega} \sup\limits_{\tilde y\in  Y}|\tilde \psi(x, \tilde y)|.
\]
For $p<\infty$, the separability of $C_{\text{per}}(Y_x)$ for   $x \in \Omega$  
and the approximation of $L^p$- functions $u \in (\Omega \to C_{\text{per}}(Y_x))$ by simple functions imply  the separability of 
$L^p(\Omega; C_{\text{per}}(Y_x))$. The norm is defined in the standard way
\[
\|\psi\|_{L^p(\Omega; C_{\text{per}}(Y_x))}:=  \int_\Omega \sup\limits_{y\in  Y_x} |\psi(x,y)|^p \, dx.
\]
 The space  $L^2(\Omega; L^2(Y_x))$ is a Hilbert space with a scalar product given by 
\begin{eqnarray*}
 \int_\Omega \int_{Y_x} \psi(x,y) \phi (x,y) \, dy dx = 
\int_\Omega \int_{Y} \tilde \psi(x,\tilde y) \tilde\phi (x,\tilde y) |\det D_x| d\tilde y dx
\end{eqnarray*}
for $\psi, \phi \in L^2(\Omega; L^2(Y_x))$ with  $\psi(x,y)= \tilde \psi(x, D_x^{-1} y)$, $\phi(x,y)= \tilde \phi(x, D_x^{-1} y)$ for a.a $x\in \Omega$ and $y \in Y_x$,
and $\tilde \psi, \tilde \phi \in L^2(\Omega; L^2(Y))$. 

Due to  the assumptions on $D$, i.e.  $D\in \text{Lip}(\overline\Omega)$ and  $|\det D(x)|$ is uniformly bounded from below and above in $\overline\Omega$,  we obtain that  
\begin{eqnarray*}
 L^2(\Omega; H^1(Y_x)) &=&\{ u \in L^2(\Omega; L^2(Y_x)), \, \nabla_y u \in  L^2(\Omega; L^2(Y_x)^d) \} \quad \qquad \text{ and }  \\
  L^2(\Omega; L^2(\partial Y_x))&=& \Big\{u:\cup_{x\in \Omega} \big(\{x\}\times \partial Y_x\big) \to \mathbb R 
\text{ measurable},  \int_\Omega \|u\|^2_{L^2(\partial Y_x)} dx< \infty\Big\}
\end{eqnarray*}
are well-defined, separable Hilbert spaces,  \cite{Dixmier, Meier, Showalter}. 

To introduce the notion of  locally-periodic two-scale convergence for a sequence $\{u^\ve\}_{\ve>0}$ in $L^2(\Omega)$ we consider the covering of $\Omega$ by cubes.  For $\ve >0$, similarly as in \cite{Briane3}, we consider the partition covering of  $\Omega$ by a family of  open non-intersecting cubes $\{\Omega_n^\ve\}_{1\leq n\leq  N_\ve}$ of side $\varepsilon^r$, with $0<r<1$, such that 
 \begin{equation*}
  \Omega\subset \cup_{n=1}^{N_\ve} \overline\Omega_n^\ve \quad \text{and } \quad 
 \Omega_n^\ve \cap \Omega \neq \emptyset ,
 \end{equation*} 
 where $N_\ve$  is the number of  $\Omega_n^\ve$ having a non-empty intersection with $\Omega$. We consider 
\[
\mathcal K^\ve= \Omega \setminus  \big(\cup_{n=1}^{\tilde N_\ve} \overline\Omega_n^\ve\big),
\]
where  by $\tilde N_\ve$  we denote  the number of all cubes $\Omega_n^\ve$ enclosed in $\Omega$. Then 
\[
\mathcal K^\ve\subset \cup_{n=\tilde N_\ve +1}^{N_\ve}\overline \Omega_n^\ve .
\]
All $\Omega_n^\ve$ with $n=\tilde N_\ve+1, \ldots, N_\ve$ have the non-empty intersection with $\partial \Omega$ and are enclosed in a $\ve^r$-neighbourhood of $\partial \Omega$. Then  the size of the domain $\mathcal K^\ve$ can be estimated 
\[
| \mathcal K^\ve| \leq  C \ve^r,
\]
with some constant $C$. This gives $\tilde N_\ve \ve^{rd} \leq |\Omega|$ and $N_\ve \ve^{rd} \leq |\Omega| + C
$  for $\ve \leq 1$. Thus
\[
N_\ve\leq C \ve^{-rd} \quad \text{ and } \quad  \tilde N_\ve \leq C \ve^{-rd}. 
\]
In the following  we shall denote   by $x_n^\ve, \tilde x_n^\ve \in \Omega_n^\ve\cap \Omega$, for $n=1,\ldots, N_\ve$, arbitrary chosen fixed points. 

We consider $\psi \in C(\overline\Omega; C_{\text{per}}(Y_x))$ and  corresponding function $\tilde \psi \in C(\overline\Omega; C_{\text{per}}(Y))$.
As a locally-periodic approximation  of $\psi$ 
we name   $\mathcal L^\ve: C(\overline\Omega; C_{\text{per}}(Y_x))\to L^\infty(\Omega)$ given by 
\[
(\mathcal L^\ve \psi)(x)=    \sum\limits_{n=1}^{N_\ve} \tilde \psi\Big(x, \frac {D^{-1}_{x_n^\ve}(x-\tilde x_n^\ve) }\ve\Big)\chi_{\Omega_n^\ve}(x)  
\quad \text{ for } x\in \Omega.
\]
We   consider also  the map $\mathcal L^\ve_0: C(\overline\Omega; C_{\text{per}}(Y_x)) \to L^\infty(\Omega)$ defined  for $x\in \Omega$ as
\[
(\mathcal L^\ve_0 \psi)(x)=\sum\limits_{n=1}^{N_\ve}  \psi\Big(x_n^\ve, \frac {x-\tilde x_n^\ve}\ve\Big)\chi_{\Omega_n^\ve}(x)=
 \sum\limits_{n=1}^{N_\ve} \tilde \psi\Big(x_n^\ve, \frac {D^{-1}_{x_n^\ve}(x-\tilde x_n^\ve)}\ve\Big)\chi_{\Omega_n^\ve}(x).
 \]
 If we choose $\tilde x_n^\ve= D_{x_n^\ve} \ve k$ for some $k \in \mathbb Z^d$, then  the periodicity of $\tilde \psi$ implies
\[
(\mathcal L^\ve \psi)(x)=   
 \sum\limits_{n=1}^{N_\ve} \tilde \psi\Big(x, \frac {D^{-1}_{x_n^\ve}x }\ve\Big)\chi_{\Omega_n^\ve}(x) \quad \text{ and } \quad 
  (\mathcal L^\ve_0 \psi)(x)=   
 \sum\limits_{n=1}^{N_\ve} \tilde \psi\Big(x_n^\ve, \frac {D^{-1}_{x_n^\ve}x }\ve\Big)\chi_{\Omega_n^\ve}(x)
\]
 for  $x\in \Omega$. In following we shall consider the case  $\tilde x_n^\ve=  D_{x_n^\ve} \ve k$, with $k\in \mathbb Z^d$,
 however all results hold for arbitrary chosen   $\tilde x_n^\ve \in \Omega_n^\ve$ with $n=1,\ldots, N_\ve$.\\
In the similar way  we define $\mathcal L^\ve\psi$ and  $\mathcal L^\ve_0\psi$ for $\psi$ in
  $C(\overline\Omega; L^q_{\text{per}}(Y_x))$ or  $L^p(\Omega; C_{\text{per}}(Y_x))$. 
\\
In the proof of  convergence theorem we shall   use the regular approximation of $\mathcal L^\ve \psi$ 
 \[
 (\mathcal L^\ve_{\rho} \psi)(x)= \sum\limits_{n=1}^{N_\ve} \tilde \psi\Big(x, \frac {D^{-1}_{x_n^\ve} x}\ve\Big)\phi_{\Omega_n^\ve}(x)  
\quad \text{ for } x\in \Omega,
\]
where  $\phi_{\O_n^\ve}$  are approximations of    $\chi_{\Omega_n^\ve}$  
such that $\phi_{\O_n^\ve} \in C^\infty_0 (\Omega_n^\ve)$ and 
\begin{equation}\label{ApproxCharactF}
 \sum\limits_{n=1}^{N_\ve}|\phi_{\O_n^\ve}  -\chi_{\Omega_n^\ve}| \to 0 \, \text{ in }  L^2(\Omega),\,   \, 
||\nabla^m \phi_{\O_n^\ve}||_{L^\infty(\mathbb R^d)}\leq C \ve^{-\rho m} \, \text{  for  } 0<r<\rho<1. 
\end{equation}
We can consider $\phi_{\Omega_n^\ve}(x)= \varphi_\ve\ast \chi_{\Omega_{n,\rho}^\ve}$  for $\O_n^\ve \subset \Omega$  and $\phi_{\Omega_n^\ve}(x)=0$  for $\Omega_n^\ve\cap \partial \Omega\neq \emptyset$, where   
$\varphi_{\ve}(x)= \frac 1{\ve^{n\rho}}\varphi(\frac x {\ve^\rho})$, with  $\varphi(x)=c\exp(-1/(1-|x|^2))$ for $|x|<1$ and $\varphi(x)=0$ for $|x|\geq 1$ and $\Omega_{n,\rho}^\ve=\{x\in \O_n^\ve, \text{dist}( x, \partial\O_n^\ve)> \ve^\rho\}$. The constant $c$ is such that $\int_{\mathbb R^d} \varphi(x) dx =1$. Then $\phi_{\O_n^\ve} \in C^\infty_0 (\Omega_n^\ve)$  and \eqref{ApproxCharactF} follow from the properties  of $\varphi_{\ve}$. Notice that 
$\|\sum_{n=1}^{N_\ve}|\phi_{\O_n^\ve}  -\chi_{\Omega_n^\ve}|\|_{L^2(\Omega)}\leq c_1 \ve^{r(d-1)} \ve^\rho N_\ve\leq c_2 \ve^{\rho-r}.$ \\
 Another  construction of $\phi_{\Omega_n^\ve}$ can be found in \cite{Briane1}.

{\it Example.}
Let $\Omega=(0,1)$, $Y=[0,1]$,  and the family of cubes $\Omega_n^\ve= (n-1, n)\ve^r$ with $\ve= N^{-1/r}$ for some $N\in \mathbb N$, and  $n=1, \ldots, N_\ve$, where  $N_\ve=N=\ve^{-r}$.  With  $D(x)=e^x$  for $x\in \overline\Omega$ we obtain the   family of intervals   $Y_x= [0,e^x]$.
 We consider $\psi(x,y)= x+\sin(2\pi e^{-x} y)$ in  $C(\overline\Omega; C_{\text{per}}(Y_x))$ and  corresponding  $\tilde \psi(x,\tilde y) = x+\sin(2\pi \tilde y)$ in $C(\overline\Omega; C_{\text{per}}(Y))$. Then the locally-periodic approximation  of $\psi$ is given by 
$\mathcal L^\ve \psi(x)= \sum_{n=1}^{N_\ve} \left( x + \sin(2\pi e^{-x_n^\ve} x/\ve)\right)\chi_{\Omega_n^\ve}(x)$, with e.g.
 $x_n^\ve=(n-1/2)\ve^r$, and   $\mathcal L^\ve_0 \psi(x)= \sum_{n=1}^{N_\ve} \left(x_n^\ve + 
\sin(2\pi e^{-x_n^\ve} x/\ve)\right)\chi_{\Omega_n^\ve}(x)$. 

\begin{definition}\label{def_two-scale}
Let  $u^\ve\in  L^2(\Omega)$ for all $\ve >0$. We say   the sequence $\{u^\ve\}$ converges locally-periodic two-scale (l-t-s) to $u \in L^2(\Omega; L^2(Y_x)) $ as $\ve \to 0$ if for any $\psi \in L^2(\Omega; C_{\text{per}}(Y_x))$ 
\begin{eqnarray*}
\lim\limits_{\ve \to 0}\int_{\Omega} u^\ve(x) \mathcal L^\ve\psi(x) dx = \int_\Omega   \ddashinttt_{Y_x}  u(x,y) \psi(x, y)  dy dx,
\end{eqnarray*}
where $\mathcal L^\ve\psi $ is the locally-periodic approximation of $\psi$.
\end{definition}

Notice that taking in the Definition \ref{def_two-scale} a test function $\psi$ independent of $y$ yields
 \begin{equation*}
  u^\ve(x) \rightharpoonup  \ddashinttt_{Y_x}  u(x,y) dy \quad \text{ in }\quad  L^2(\Omega).
\end{equation*}
The subsequent convergence results for the locally-periodic approximation will  ensure  that Definition \ref{def_two-scale} does not depend on the choice of  $x_n^\ve, \tilde x_n^\ve \in \Omega_n^\ve$ for $n=1, \ldots, N_\ve$. 
 
The main results of this section are the compactness theorem for the locally-periodic  two-scale convergence and the characterisation of  the locally-periodic two-scale limit for  a sequence bounded in $H^1(\Omega)$.
\begin{theorem}\label{L2_lp_covregence}
Let $\{u^\ve\}$ be a bounded sequence in $L^2(\Omega)$. Then there exists a subsequence  of $\{u^\ve\}$, denoted again by $\{u^\ve\}$, and a  function $u \in L^2(\Omega; L^2(Y_x))$, such that  $u^\ve \to u$ in the locally-periodic two-scale sense as $\ve \to 0$. 
\end{theorem} 
\begin{theorem}\label{compactness_theorem}
 Let $\{u^\ve\} $ be a bounded sequence in $H^1(\Omega)$ that converges weakly to $u$ in $H^1(\Omega)$. Then
\begin{itemize}
\item[$\circ$]   the sequence $\{u^\ve\}$ converges locally-periodic two-scale to $u$;
\item[$\circ$]  there exist a subsequence of $\{\nabla u^\ve \}$,  denoted again by $\{\nabla u^\ve \}$, and a function  $u_1\in L^2(\Omega; H^1_{\text{per}}(Y_x)/\mathbb R)$  such that $\nabla u^\ve \to \nabla u + \nabla_y u_1$ in the locally-periodic two-scale sense  as $\ve \to 0$. 
 \end{itemize}
\end{theorem}

The proofs of the main  results  rely on some convergence results for  locally-periodic approximations of   $Y_x$-periodic  functions.
\begin{lemma}\label{converg_ost}
For   $\psi \in  L^p(\Omega; C_{\text{per}}(Y_x))$, with  $1\leq p <\infty$, we have 
\begin{eqnarray}\label{LocPeriodConver}
\lim\limits_{\ve \to 0} \int_{\Omega}\left|\mathcal L^\ve \psi (x) \right|^p  dx = 
 \int_\Omega \, \ddashinttt_{Y_x} |\psi(x, y)|^p dy\,  dx
\end{eqnarray}
and for  $\psi \in  L^1(\Omega; C_{\text{per}}(Y_x))$ we obtain 
\begin{eqnarray}\label{LocPeriodConver_L1}
\lim\limits_{\ve \to 0} \int_{\Omega} \mathcal L^\ve \psi (x)  dx = 
 \int_\Omega \, \ddashinttt_{Y_x} \psi(x, y) dy\,  dx.
\end{eqnarray}
For $\psi \in C(\overline\Omega; L^p_{\text{per}}(Y_x))$, where  $1\leq p <\infty$, we have   
 \begin{eqnarray}\label{EstimOscil2}
\lim\limits_{\ve \to 0}\int_{\Omega} \, \left|\mathcal L^\ve_0 \psi (x) \right|^p dx =  \int_\Omega\,  \ddashinttt_{Y_x} |\psi(x, y)|^p dy\,  dx
\end{eqnarray}
and for $ \psi \in C(\overline\Omega; L^1_{\text{per}}(Y_x))$ we obtain 
 \begin{eqnarray}\label{EstimOscil2_L1}
\lim\limits_{\ve \to 0}\int_{\Omega} \,  \mathcal L^\ve_0 \psi (x) dx =  \int_\Omega\,  \ddashinttt_{Y_x} \psi(x, y)  dy\,  dx.
\end{eqnarray}
If $\psi \in L^p(\overline\Omega; C^1_{\text{per}}(Y_x))$,  with $1\leq p <\infty$, then we have 
\begin{eqnarray}\label{grad_loc}
 \lim\limits_{\ve \to 0} \int_\Omega\,  \left|\mathcal L^\ve \nabla_y \psi (x)\right|^p  dx= \int_\Omega \,   \ddashinttt_{Y_x} |\nabla_y \psi(x, y)|^p  dy dx.
\end{eqnarray}
 \end{lemma}
 
\begin{proof}  We start with the proof of \eqref{LocPeriodConver} for $\psi \in C(\overline\Omega; C_{\text{per}}(Y_x))$.
The translations of  $Y=[0,1]^d$ are given by $Y^i= Y+ k_i$ with $k_i \in \mathbb Z^d $. Then for   $ x_n^\ve \in \Omega_n^\ve$   and $D_{x_n^\ve}=D(x_n^\ve)$ we consider   $Y_{x_n^\ve}^i= D_{x_n^\ve} Y^i$, obtained by the linear transformation of  $Y^i$,  and 
cover   $\Omega_n^\ve$, with $1\leq n \leq N_\ve$, by the family of closed parallelepipeds $\{\ve Y_{x_n^\ve}^i\}_{i=1}^{I_n^\ve}$:
\[
\Omega_n^\ve \subset \cup_{i=1}^{I_{n}^\ve} \ve  Y_{x_n^\ve}^i \quad \text{ such that } \,\,  \ve Y_{x_n^\ve}^i \cap \overline\Omega_n^\ve \neq \emptyset,
\] 
where  $I_n^\ve$ is  the number of parallelepipeds covering  $\Omega_n^\ve$. The  domain $\mathcal M_n^\ve$ given by
\[
\mathcal M_n^\ve = \Omega_n^\ve \setminus \Big( \cup_{i=1}^{\tilde I_n^\ve} \ve  Y_{x_n^\ve}^i \Big), 
\]
with $\tilde I_n^\ve$ as the number of all  $\ve Y_{x_n^\ve}^i$ enclosed in $\Omega_n^\ve$, satisfies 
 $ \mathcal M_n^\ve \subset \cup_{i=\tilde I_n^\ve+1}^{I_{n}^\ve} \ve  Y_{x_n^\ve}^i. $
All  $\ve  Y_{x_n^\ve}^i$, for $i=\tilde I_n^\ve+1,\ldots, I_n^\ve$, have the non-empty intersection with $\partial \Omega_n^\ve$ and are enclosed in a $\ve$-neighbourhood of $\partial \Omega_n^\ve$.  This ensures the estimate
$$|\mathcal M_n^\ve| \leq  C_1|\partial \Omega_n^\ve| \ve\leq  C_2 \, \ve^{(d-1)r+1} \qquad \, \text{ for all } \, \, n=1,\ldots, N_\ve. $$ 
Thus we have 
$ \ve^d |\det D_{x_n^\ve}| \tilde I_n^\ve \leq \ve^{dr}$ and $\ve^{d} |\det D_{x_n^\ve}| I_n^\ve \leq C \ve^{dr}$ for $\ve \leq 1$, i.e.
$$
\tilde I_n^\ve \leq C\ve^{d(r-1)} \quad \text{ and } \quad I_n^\ve \leq C \ve^{d(r-1)}\quad \text{ for all } \, \, n=1,\ldots, N_\ve.
$$
Considering the covering of $\Omega_n^\ve$, we rewrite the integral on  the left hand side of \eqref{LocPeriodConver}  
\begin{eqnarray*}
  \int_{\Omega} \left|\mathcal L^\ve \psi(x)\right|^p dx 
&  =& \sum\limits_{n=1}^{\tilde N_\ve}  \sum\limits_{i=1}^{\tilde I_n^\ve} \int_{\ve Y_{x_n^\ve}^i} 
\Big|\tilde \psi\Big(x_n^\ve, \frac {D^{-1}_{x_n^\ve}x}\ve\Big)\Big|^p dx \\
&+& \sum\limits_{n=1}^{\tilde N_\ve} \sum\limits_{i=1}^{\tilde I_n^\ve}  \int_{\ve Y_{x_n^\ve}^i} 
\Big(\Big|\tilde \psi\Big(x,  \frac{D^{-1}_{x_n^\ve}x} \ve\Big)\Big|^p-\Big|\tilde \psi\Big(x_n^\ve,  \frac {D^{-1}_{x_n^\ve}x} \ve\Big) \Big|^p \Big)dx \\
&+&
 \sum\limits_{n=1}^{\tilde N_\ve} \int_{\mathcal M_n^\ve} \Big|\tilde\psi\Big(x, \frac  {D^{-1}_{x_n^\ve}x}\ve\Big)\Big|^p dx
+  \int_{\mathcal K^\ve}   \left|\mathcal L^\ve \psi(x)\right|^p dx= I_1+I_2+I_3+I_4. 
\end{eqnarray*}  
Applying  the inequality $||a|^p-|b|^p| \leq p ( |a|^{p-1}+|b|^{p-1})|a-b|$, see \cite{Alt}, the assumptions on $D$, the continuity $\psi$, the bounds for $\tilde N_\ve$ and $\tilde I_n^\ve$, and the property 
 $$\sup\limits_{1\leq i\leq \tilde I_n^\ve}\sup\limits_{x\in \ve Y_{x_n^\ve}^i}|x-x_n^\ve| \leq C \ve^{r}, $$  
ensured by the fact that $|\Omega_n^\ve|=\ve^r$ for all $1\leq n \leq N_\ve$,  we can estimate  $I_2$   by
\begin{equation}\label{Estim_Contin}
2p \| \psi\|_{C(\overline\Omega, C_{\text{per}}(Y_x))}^{p-1}\sum\limits_{n=1}^{\tilde N_\ve}  \sum\limits_{i=1}^{\tilde I_n^\ve} |\ve  Y_{x_n^\ve}^i|\sup\limits_{x\in \ve Y_{x_n^\ve}^i}\sup\limits_{\tilde y\in  Y}
|\tilde \psi(x,\tilde y) -\tilde \psi(x_n^\ve, \tilde y)| 
\leq C \delta_1(\ve) |\Omega|, 
\end{equation}
where $ \delta_1(\ve)\to 0$ as $\ve \to 0$. For  $I_3$ and $I_4$
  the regularity of $\psi$ together with the bounds for $|\mathcal K^\ve|$,   $|\mathcal M_n^\ve|$, and  $\tilde N_\ve$ implies 
\begin{equation}\label{Estim_rest} 
|I_3|+|I_4|\leq \|\psi(x,y)\|_{C(\overline\Omega, C_{\text{per}}(Y_x))}^p\Big( \sum\limits_{n=1}^{\tilde N_\ve} |\mathcal M_n^\ve | + |\mathcal K^\ve|\Big) \leq  C(\ve^{1-r}+ \ve^r).
\end{equation}
Combining now \eqref{Estim_Contin} and \eqref{Estim_rest} gives
\begin{eqnarray}\label{Estim_Int_LocPerApprox} 
\qquad  \int_{\Omega}\left|\mathcal L^\ve \psi(x)\right|^p dx \hspace{-0.1 cm }
&=&  \hspace{-0.1 cm } \sum\limits_{n=1}^{\tilde N_\ve} \tilde I_n^\ve \ve^d  |\det D_{x_n^\ve}| \int_{Y} |\tilde \psi(x_n^\ve, \tilde y)|^p d\tilde y
+ \delta(\ve) \\
 \hspace{-0.1 cm } &=& \hspace{-0.1 cm }\int_\Omega  \sum\limits_{n=1}^{\tilde N_\ve}  \ddashinttt_{Y_{x_n^\ve}} |\psi(x_n^\ve, y)|^p dy \, \chi_{\Omega_n^\ve}(x)\,  dx 
 -\sum\limits_{n=1}^{\tilde N_\ve}|\mathcal M_n^\ve|  \ddashinttt_{Y_{x_n^\ve}} |\psi(x_n^\ve, y)|^p dy+ \delta(\ve),\nonumber
\end{eqnarray}
where $\delta(\ve) \to 0$ as $\ve\to 0$. 
 The  continuity of $\psi$ with respect to $x$ ensures
\begin{equation}\label{estimM_n_ve}
\sum\limits_{n=1}^{\tilde N_\ve}|\mathcal M_n^\ve|  \ddashinttt_{Y_{x_n^\ve}} |\psi(x_n^\ve, y)|^p dy
\leq C_1 \ve^{1-r} \sup\limits_{x\in \Omega} \int_{ Y } |\tilde \psi(x, \tilde y)|^p d\tilde y \leq C_2 \ve^{1-r}.
\end{equation} 
Thus, the continuity of  $F$, given by  $F(x)= \ddashinttt_{Y_{x}} |\psi(x, y)|^p dy $, in $\overline\Omega$ and  the limit as $\ve \to 0$ in  \eqref{Estim_Int_LocPerApprox} provide   convergence \eqref{LocPeriodConver} for $\psi \in C(\overline\Omega; C_{\text{per}}(Y_x))$. 

To  prove  \eqref{LocPeriodConver} for  $\psi\in L^p(\Omega; C_{\text{per}}(Y_x))$  we consider an approximation of  $\psi $ by a sequence 
$\{\psi_m \}\subset C(\overline\Omega; C_{\text{per}}(Y_x))$ such that   $\psi_m \to \psi$  in $L^p(\Omega; C_{\text{per}}(Y_x))$.  Using  
\begin{equation*}
\begin{array}{l}
   \left|\tilde \psi\Big( x,  \frac {D^{-1}_{x_n^\ve} x} \ve\Big)\right|^p-\left|\tilde \psi_m\Big(x,  \frac  {D^{-1}_{x_n^\ve} x} \ve\Big) \right|^p  =
 \left|\tilde \psi\Big( x,  \frac {D^{-1}_x x} \ve\Big)\right|^p-\left|\tilde \psi_m\Big(x,   \frac {D^{-1}_x x} \ve\Big) \right|^p\\
+\left|\tilde \psi\Big( x,  \frac {D^{-1}_{x_n^\ve} x} \ve\Big)\right|^p-\left|\tilde \psi\Big(x,   \frac {D^{-1}_x x} \ve\Big) \right|^p
+ \left|\tilde \psi_m\Big(x,   \frac {D^{-1}_x x} \ve\Big) \right|^p-\left|\tilde \psi_m\Big( x,  \frac {D^{-1}_{x_n^\ve} x} \ve\Big)\right|^p
\end{array}
\end{equation*}
we obtain the following estimate
\begin{eqnarray}\label{EstimApproxCont}
 && \int_{\Omega}  
\left|\left|\mathcal L^\ve \psi( x)\right|^p-\left|\mathcal L^\ve \psi_m (x) \right|^p \right| d x  
 \leq 
\int_{\Omega} \sup\limits_{y\in Y_x} 
\left |\, | \psi(x,   y)|^p-| \psi_m(x, y) |^p \right |d x \nonumber\\
&&\hspace{1.0 cm} \qquad +
\int_{\Omega} \sum\limits_{n=1}^{N_\ve} \sup\limits_{\tilde y \in Y} 
\left|\left|\tilde \psi\left( x,  \tilde y \right)\right|^p-\left|\tilde \psi\left(x,   D^{-1}_xD_{x_n^\ve}\tilde y \right) \right|^p \right| 
\chi_{\Omega_n^\ve}(x)d x\\
&& \hspace{1.0 cm} \qquad +
\int_{\Omega} \sum\limits_{n=1}^{N_\ve} 
\sup\limits_{\tilde y \in Y} 
\left|\left|\tilde \psi_m\left( x,   \tilde y \right)\right|^p-\left|\tilde \psi_m\left(x,   D^{-1}_x D_{x_n^\ve}\tilde y \right) \right|^p \right|\chi_{\Omega_n^\ve}(x)d x.\nonumber
\end{eqnarray}
 The inequality $||\psi|^p-|\psi_m|^p| \leq p ( |\psi|^{p-1}+|\psi_m|^{p-1})|\psi-\psi_m|$, see \cite{Alt},  together with the H\"older  inequality and 
 the boundedness of $\psi_m$ and $\psi$ in $L^p(\Omega; C_{\text{per}}(Y_x))$ implies  
  \begin{equation*}
  \int_{\Omega} \sup\limits_{y\in Y_x} 
\left |\, | \psi |^p-| \psi_m |^p \right |d x \leq 
C \Big(\int_{\Omega} \sup\limits_{y\in Y_x}  |\psi-\psi_m|^p dx \Big)^{\frac 1p}.
  \end{equation*}  
The assumptions on $D$ and  $\Omega^\ve_n$ ensure   $|\tilde y - D^{-1}_x D_{x_n^\ve}\tilde y|\leq \delta_2(\ve)$ for $x \in \Omega^\ve_n$, $\tilde y\in Y$ and  $n=1,\ldots, N_\ve$, where $\delta_2(\ve) \to 0$ as $\ve \to 0$.
Thus, using   the continuity of  $\psi$ and $\psi_m$ with respect to $y$, the  convergence of $\psi_m$,
 and taking in      \eqref{EstimApproxCont} the limit  as $\ve \to 0$ and then as $m\to \infty$  we conclude
\begin{equation*}
  \lim\limits_{m\to \infty}\lim\limits_{\ve \to 0}  \int_{\Omega}\left(
  \left|\mathcal L^\ve \psi (x)\right|^p-\left|\mathcal L^\ve\psi_m (x) \right|^p \right)d x  =0.
\end{equation*}  
Applying  now   the calculations from above  to $\psi_m\in C(\overline\Omega; C_{\text{per}}(Y_x))$ for $m \in \mathbb N$ and considering the strong convergence of $\psi_m$ we obtain for    $\psi\in L^p(\Omega; C_{\text{per}}(Y_x))$
\begin{eqnarray*}
  \lim\limits_{\ve \to 0}\int_{\Omega} \left|\mathcal L^\ve \psi (x)\right|^p dx &=& 
\lim\limits_{m\to \infty}\lim\limits_{\ve \to 0} \int_{\Omega} \left|\mathcal L^\ve\psi_m (x) \right|^p d x
\\
&=&\lim\limits_{m\to \infty} \int_\Omega \, \ddashinttt_{Y_x} |\psi_m(x, y)|^p dy\,  dx
= \int_\Omega \, \ddashinttt_{Y_x} |\psi(x, y)|^p dy\,  dx.
 \end{eqnarray*}  
 
The  proof for  \eqref{LocPeriodConver_L1} follows the same lines as for \eqref{LocPeriodConver}.  First we consider  $\psi \in C(\overline\Omega; C_{\text{per}}(Y_{x}))$ and   write  the integral on the left hand side of \eqref{LocPeriodConver_L1} in the form 
\begin{eqnarray*}
  \int_{\Omega} \mathcal L^\ve \psi(x)  dx 
&  =& \sum\limits_{n=1}^{\tilde N_\ve}  \sum\limits_{i=1}^{\tilde I_n^\ve} \int_{\ve Y_{x_n^\ve}^i} 
\tilde \psi\Big(x_n^\ve, \frac {D^{-1}_{x_n^\ve}x}\ve\Big) +
\Big[\tilde \psi\Big(x,  \frac{D^{-1}_{x_n^\ve}x} \ve\Big)-\tilde \psi\Big(x_n^\ve,  \frac {D^{-1}_{x_n^\ve}x} \ve\Big)  \Big] dx\\
&+&
 \sum\limits_{n=1}^{\tilde N_\ve} \int_{\mathcal M_n^\ve} \tilde\psi\Big(x, \frac  {D^{-1}_{x_n^\ve}x}\ve\Big) dx
+  \int_{\mathcal K^\ve}  \mathcal L^\ve \psi(x)  dx.
\end{eqnarray*}  
Using  estimates  \eqref{Estim_Contin}, \eqref{Estim_rest} and  \eqref{estimM_n_ve} with $p=1$ we can  conclude that 
\begin{eqnarray}\label{Estim_Int_LocPerApprox_L1} 
 \int_{\Omega} \mathcal L^\ve \psi(x) dx 
&=& \sum\limits_{n=1}^{\tilde N_\ve} |\Omega_n^\ve|\,  \ddashinttt_{Y_{x_n^\ve}} \psi(x_n^\ve, y) dy + \delta(\ve),\nonumber
\end{eqnarray}
where $\delta(\ve) \to 0$ as $\ve\to 0$. 
Thus, the continuity of  $F(x)= \ddashinttt_{Y_{x}} \psi(x, y) dy $ in $\overline\Omega$ and  the limit as $\ve \to 0$ 
 provide   convergence \eqref{LocPeriodConver_L1} for  $\psi \in C(\overline\Omega; C_{\text{per}}(Y_{x}))$. 
To show \eqref{LocPeriodConver_L1} for $\psi \in L^1(\Omega; C_{\text{per}}(Y_x))$ we  approximate $\psi$ by $\{\psi_m\} \subset C(\overline\Omega; C_{\text{per}}(Y_x))$. Then  similar calculations as in the proof of \eqref{LocPeriodConver}   ensure the convergence \eqref{LocPeriodConver_L1}.

In the same  way as above we obtain the equality 
\begin{eqnarray}\label{ConverY_x2}
\int_{\Omega}  \left|\mathcal L^\ve_0 \psi (x)\right|^p dx &=& \sum\limits_{n=1}^{\tilde N_\ve}  \tilde I_n^\ve \ve^{d}   |\det D_{x_n^\ve}| \int_{ Y } |\tilde \psi(x_n^\ve, \tilde y)|^p d\tilde y\quad 
 \\ && +\sum\limits_{n=1}^{\tilde N_\ve}\int_{\mathcal M_n^\ve}\Big|
\tilde \psi\Big(x_n^\ve, \frac{D^{-1}_{x_n^\ve}{x}}\ve \Big)\Big|^p dx +
\int_{\mathcal K^\ve} \left|\mathcal L^\ve_0\psi(x)\right|^p  dx.\nonumber
\end{eqnarray}
The second and  third integrals can be estimated by
\begin{eqnarray*}
 \sum_{n=1}^{\tilde N_\ve}  \sum_{i=\tilde I_n^\ve+1}^{I^\ve_n} \ve^d |\det D_{x_n^\ve}|\int_{Y} |\tilde \psi(x_n^\ve, \tilde y)|^p d\tilde y 
&\leq & C \ve^{1-r} \sup\limits_{x\in\Omega}  \int_{Y_x} |\psi(x, y)|^p d y , \\ 
\sum_{n=\tilde N_\ve+1}^{N_\ve} I^\ve_n\ve^d  |\det D_{x_n^\ve}|
\int_{Y} |\tilde \psi(x_n^\ve, \tilde y)|^p d\tilde y  &\leq & C \ve^{r} \sup\limits_{x\in\Omega}  \int_{Y_x} | \psi(x,  y)|^p d y.
\end{eqnarray*}
Then, using    \eqref{estimM_n_ve} and passing in  \eqref{ConverY_x2} to the limit as $\ve \to 0$
yield  convergence \eqref{EstimOscil2}.

Similar  arguments  imply also the convergence \eqref{EstimOscil2_L1}. 

Applying convergence  \eqref{LocPeriodConver}  to $\nabla_y \psi \in L^p(\overline\Omega; C_{\text{per}}(Y_x)^d)$  gives
\begin{eqnarray*}
\lim\limits_{\ve \to 0}  \int_{\Omega}
\left|\mathcal L^\ve \nabla_y \psi(x)\right|^p dx &=& \lim\limits_{\ve \to 0} \int_\Omega\,\sum\limits_{n=1}^{N_\ve}  
\Big|D_{x_n^\ve}^{-T}\nabla_{\tilde y} \tilde \psi\Big(x, \frac {D_{x_n^\ve}^{-1}x} \ve\Big)\Big|^p \chi_{\Omega_n^\ve} dx \nonumber\\
&=&\int_\Omega  \ddashinttt_{Y_x} |\nabla_{y} \psi (x,y)|^p dy  dx
=\int_\Omega \ddashinttt_{Y} \big| D^{-T}_x\nabla_{\tilde y} \tilde \psi (x, \tilde y)\big|^p d\tilde y    dx,
\end{eqnarray*}
where $D^{-T}_{x}$ is the transpose of the matrix $D^{-1}_{x}$.
\hfill \end{proof}

To proof the Lemma for an arbitrary chosen  $\tilde x_n^\ve\in \Omega_n^\ve$ we shall consider the shifted covering $\Omega_n^\ve \subset \tilde x_n^\ve+ \cup_{i=1}^{I_{n}^\ve} \ve  Y_{x_n^\ve}^i$ with the same properties as above. Then for $x\in \Omega_n^\ve$ we have $x-\tilde x_n^\ve \in \ve  Y_{x_n^\ve}^i$ for some $1\leq i \leq I_n^\ve$ and, applying the change of variables $D^{-1}_{x_n^\ve}(x-\tilde x_n^\ve)/\ve= \tilde y$, we can conduct the same calculations as  for $\tilde x_n^\ve=D_{x_n^\ve} \ve k$, $k\in \mathbb Z^d$.

We consider $\tilde f\in L^p_{\text{per}}(Y)$ for  $1< p \leq \infty$ and define  $f(x,y):=\tilde f(D^{-1}_x y)$ for $x \in \Omega$ and a.a. $y\in Y_x$. Applying similar calculations as in Lemma \ref{converg_ost}  we can show  the  boundedness in $L^p$ of  $\{\mathcal L^\ve f\}$, which due  to the structure of $f$ coincides with
 $\{\mathcal L^\ve_0 f\}$,
 \[
\|\mathcal L^\ve f\|^p_{L^p(\Omega)} \leq \sum\limits_{n=1}^{N_\ve} I^\ve_n \ve^d  |\det  D_{x_n^\ve}| 
\|\tilde f(\tilde y) \|^p_{L^p(Y)}\leq C \| \tilde f \|^p_{L^p(Y)},
\]
and the convergence 
\begin{equation*}
\lim\limits_{\ve \to 0}\int_\Omega \mathcal L^\ve f(x) dx=\lim\limits_{\ve \to 0}\sum\limits_{n=1}^{\tilde N_\ve}  \tilde I_n^\ve \ve^d
|\det  D_{x_n^\ve}| \int_{Y} \tilde f(\tilde y) d\tilde y=\int_{\Omega} \ddashinttt_{Y_x}  f(x, y) d y dx.
 \end{equation*}
Thus, in the same manner as for periodic functions we obtain   as $\ve \to 0$
\begin{eqnarray*}
\mathcal L^\ve f(x) \rightharpoonup  \ddashinttt_{Y_x}f(x,y) dy \quad   \text{ weakly \,  in }  \,  L^p(\Omega) \text{ for } p>1, \quad 
 \text{ weakly}\ast \text{ in }  \,  L^\infty(\Omega).
\end{eqnarray*}

{\it Remark. }  As we can see  from the proofs, all convergence results are independent of the choice of  the points
$x_n^\ve, \tilde x_n^\ve \in \Omega_n^\ve$ for $n=1, \ldots, N_\ve$.

Now we  prove the compactness theorem for locally-periodic two-scale convergence. 
\begin{proof} ({\it Theorem \ref{L2_lp_covregence}} )
We  shall apply similar ideas as for  the two-scale convergence with periodic test functions, \cite{Allaire, Lukkassen}.

For   $\psi \in L^2(\Omega; C_{\text{per}}(Y_x))$ we consider a  functional $\mu^\ve$ given by
\[
\mu^\ve(\psi)= \int_{\Omega} u^\ve(x) \mathcal L^\ve    \psi(x)  dx.
\]
Using the Cauchy-Schwarz inequality and the boundedness of $\{u^\ve \} $ in $L^2(\Omega)$,
we obtain that    $\mu^\ve$ is a bounded linear functional  in $(L^2(\Omega;  C_{\text{per}}(Y_x)))^\prime$:
\begin{equation*}
|\mu^\ve(\psi)| \leq \|u^\ve\|_{L^2(\Omega)} \| \mathcal L^\ve\psi \|_{L^2(\Omega)} 
\leq C ||\psi||_{L^2(\Omega; C_{\text{per}}(Y_x))}.
\end{equation*}
 Since $L^2(\Omega;  C_{\text{per}}(Y_x))$ is separable,  there exists $\mu_0 \in (L^2(\Omega;  C_{\text{per}}(Y_x)))^\prime$ 
such that, up to a subsequence,  $\mu^\ve \stackrel{\ast }{\rightharpoonup} \mu_0$ in $(L^2(\Omega;  C_{\text{per}}(Y_x)))^\prime$.
Using convergence  \eqref{LocPeriodConver}  and assumptions on $D$ yields 
$$
|\mu_0(\psi)|= |\lim\limits_{\ve\to 0} \mu^\ve(\psi)|\leq C_1\lim\limits_{\ve\to 0}
 \|\mathcal L^\ve\psi \|_{L^2(\Omega)}    \leq C_2 ||\psi ||_{L^2(\Omega; L^2(Y_x))}.
$$
Thus, $\mu_0$ is a bounded linear functional on the Hilbert space $L^2(\Omega; L^2(Y_x))$. The definition of  $L^2(\Omega; L^2(Y_x))$, the density of   
$L^2(\Omega; C_{\text{per}}(Y_x))$ in $L^2(\Omega; L^2(Y_x))$,   and the Riesz representation theorem  imply the existence of 
  $\tilde u\in L^2(\Omega;L^2(Y_x))$ such that 
$$
\mu_0(\psi)= \int_\Omega  \int_{Y_x} \tilde u(x,y) \psi(x,y) dy dx=\int_\Omega  \ddashinttt_{Y_x} u(x,y) \psi(x,y) dy dx,
$$
with $u = \tilde u |Y_x|$ and $u\in L^2(\Omega;L^2(Y_x))$. Therefore,  there exists a  subsequence of $\{u^\ve\}$  that converges locally-periodic two-scale to $u \in  L^2(\Omega;L^2(Y_x))$. \hfill
\end{proof}

Similar as for the two-scale convergence, see \cite{Lukkassen}, we can show that it is sufficient to consider  more regular  test functions in the definition of locally-periodic two-scale convergence, assuming that the sequence is bounded in $L^2(\Omega)$. 
\begin{proposition}\label{ConverSmoothL21}
Let $\{u^\ve\}$ be a bounded sequence  in $L^2(\Omega)$,  such that 
\begin{equation}\label{ConverSmooth}
\lim\limits_{\ve \to 0}\int_{\Omega} u^\ve(x)
\mathcal L^\ve \phi(x) dx  =
\int_\Omega   \ddashinttt_{Y_x}  u(x,y) \phi(x, y)  dy dx
\end{equation}
for every $\phi \in W^{1,\infty}_0(\Omega; C^\infty_{\text{per}}(Y_x))$. 
Then $\{u^\ve\}$ converges l-t-s  to $u$.
\end{proposition}
\begin{proof}
We consider  $\psi \in L^2(\Omega; C_{\text{per}}(Y_x))$ and  $\{\phi_m \}\subset W^{1,\infty}_0(\Omega; C^\infty_{\text{per}}(Y_x))$
such that $\phi_m \to \psi$ in $ L^2(\Omega; C_{\text{per}}(Y_x))$ as $m \to \infty$.
Then we can write
\begin{equation*}
\lim\limits_{\ve \to 0}\int_{\Omega} u^\ve(x)
\mathcal L^\ve\psi (x) dx=
 \lim\limits_{m\to \infty}
\lim\limits_{\ve \to 0}\int_{\Omega} \left[u^\ve(x)  \mathcal L^\ve \phi_m(x) + u^\ve(x) (\mathcal L^\ve \psi- \mathcal L^\ve\phi_m)\right] dx.   \end{equation*}

 Assumed convergence \eqref{ConverSmooth}  and the   convergence of $\phi_m$  to $\psi$ ensure 
   \begin{equation*}
\lim\limits_{m\to \infty}\lim\limits_{\ve \to 0} \int_{\Omega} u^\ve(x)  \mathcal L^\ve \phi_m(x) dx=
\lim\limits_{m\to \infty}\int_\Omega   \ddashinttt_{Y_x}  u \, \phi_m \, dy dx
=\int_\Omega   \ddashinttt_{Y_x}  u(x,y) \psi(x, y)  dy dx.
  \end{equation*}
The boundedness of $\{u^\ve\}$ in $L^2(\Omega)$,  estimates  similar  to   \eqref{EstimApproxCont} in
 the proof of   Lemma~\ref{converg_ost},  the continuity of $\psi$ and $\phi_m$ with respect 
to  the second variable, the  convergence of $\phi_m$, and the regularity of $D$ imply  that 
the second term on the right hand side converges to zero as $\ve \to 0$ and   $m\to \infty$.  Thus,  due to the arbitrary choice  of  $\psi \in L^2(\Omega; C_{\text{per}}(Y_x))$, we  conclude  that  $u^\ve \to u$ in the locally-periodic two-scale sense as $\ve \to 0$. 
\hfill  \end{proof}
 
 If we  assume more regularity on the transformation matrix $D$,  the result in Proposition~\ref{ConverSmoothL21} holds also
for more regular, with respect to $x$,  test functions $\phi$. 

In the following we prove a technical lemma on the strong $(H^{1}(\Omega))^\prime-$convergence, which will be used in the proof of 
Theorem~\ref{compactness_theorem}.
\begin{lemma}\label{H_1_conv}
 For   $\Psi \in W^{1,\infty}_0(\Omega; C^1_{\text{per}}(Y_x)^d)$, corresponding  $\tilde \Psi \in W^{1,\infty}_0(\Omega; C^1_{\text{per}}(Y)^d)$, and  $\phi_{\Omega_n^\ve} \in C^\infty_0(\Omega_n^\ve)$ such that 
\begin{equation}\label{estim_grad_phi}
||\nabla^m \phi_{\Omega^\ve_n}||_{L^\infty(\mathbb R^d)} \leq C\ve^{-m \rho}  \qquad \text{with } \, \, 0<\rho<1,
\end{equation}
 for  $n=1,\ldots,  N_\ve$,  holds
\begin{eqnarray}\label{convergence_H_1}
 \sum\limits_{n=1}^{N_\ve}  \Big(\tilde\Psi\Big(\cdot, D_{x_n^\ve}^{-1}\frac{ \cdot} \ve\Big)
-  \ddashinttt_{Y_x} \Psi(\cdot, y) d y \Big) \nabla \phi_{\Omega_n^\ve}  \to 0  \quad \text{ in } \, \, (H^{1}(\Omega))^\prime.
\end{eqnarray}
In particular,   for $x_n^\ve\in \Omega_n^\ve$ and  $1\leq  n \leq N_\ve$ we have 
\begin{eqnarray}\label{convergence_H_1_2}
 \sum\limits_{n=1}^{N_\ve}  \Big(\Psi\Big(x_n^\ve, \frac{ \cdot} \ve\Big)
-  \ddashinttt_{Y_{x_n^\ve}} \Psi(x_n^\ve, y) d y \Big) \nabla \phi_{\Omega_n^\ve}    \to 0  \quad \text{ in } \, \, (H^{1}(\Omega))^\prime.
\end{eqnarray}
\end{lemma}

\begin{proof}
For  $ \Psi \in  W^{1,\infty}_0(\Omega; C^1_{\text{per}}(Y_x)^d) $ there exists a unique $h_1 \in W^{1,\infty}_0(\Omega; C^2_{\text{per}}(Y)^d)$ with zero $Y$-mean value  such that 
 \begin{equation}\label{Eq_h}
 \Delta_{\tilde y} h_1(x,\tilde y)= \tilde \Psi(x, \tilde y)- \ddashinttt_{Y_x} \Psi(x,    y) d  y. 
 \end{equation}
Considering   \eqref{Eq_h} with   $\tilde y=D^{-1}_{x_n^\ve} x/\ve$  we obtain  for a.a. $x\in \Omega_n^\ve\cap \Omega$ and  $1\leq n \leq N_\ve$
\begin{eqnarray}\label{H1_sec1}
\tilde\Psi\Big(x, \frac{ D_{x_n^\ve}^{-1}x} \ve\Big)- \ddashinttt_{Y_x}\Psi(x, y) dy
&=& \ve \nabla\cdot\Big( D_{x_n^\ve}\cdot\nabla_{\tilde y}  h_1\Big(x, \frac{D^{-1}_{x_n^\ve} x}\ve\Big) \Big) \nonumber \\
&& -\ve
\nabla_x\cdot\Big( D_{x_n^\ve}\cdot\nabla_{\tilde y}  h_1\Big(x, \frac{D^{-1}_{x_n^\ve} x}\ve\Big)\Big).
\end{eqnarray}
Similarly as in   \cite{Briane3}, we define $F_1(x,\tilde y)=\nabla_{\tilde y}  h_1(x,\tilde y)$ with 
$F_1 \in W^{1,\infty}_0(\Omega; C^1_{\text{per}}(Y)^{d\times d})$ and 
  $F^{n}_1(x)=  F_1(x, D^{-1}_{x_n^\ve} x/\ve)$ for $x\in \Omega^\ve_n\cap \Omega$.  Applying \eqref{H1_sec1}, we have for a.a. $x\in \Omega$
\begin{eqnarray}\label{H1_Conver22}
 && \sum\limits_{n=1}^{N_\ve} \Big[ \tilde \Psi\Big(x, \frac {D_{x_n^\ve}^{-1} x} \ve\Big) - \ddashinttt_{Y_x}\Psi dy\Big]
\nabla \phi_{\Omega_n^\ve} =\sum\limits_{n=1}^{N_\ve} \ve \left[\nabla  \cdot \left(D_{x_n^\ve}F^n_1 \right) - 
 \nabla_x  \cdot\left( D_{x_n^\ve}F^n_1 \right) \right]\nabla \phi_{\Omega_n^\ve} \nonumber \\
 &&=\sum\limits_{n=1}^{N_\ve}\Big[ \nabla\cdot \left(\ve (D_{x_n^\ve}  F^{n}_1)\nabla\phi_{\Omega_n^\ve} \right) 
 - \ve   \nabla_x \cdot \left(D_{x_n^\ve}F^{n}_1\right)  \nabla \phi_{\Omega_n^\ve} - \ve D_{x_n^\ve} F_1^{n} : \nabla^2 \phi_{\Omega_n^\ve}\Big].
\end{eqnarray}

The  boundedness of $\sum_{n=1}^{N_\ve}D_{x_n^\ve}F^{n}_1\chi_{\Omega_n^\ve}$   and 
$\sum_{n=1}^{N_\ve}\nabla_x \cdot \big(D_{x_n^\ve}F^{n}_1\big)\chi_{\Omega_n^\ve}$ in  $L^2(\Omega)$, assured by the regularity of  $F_1$ and $D$,  and assumption \eqref{estim_grad_phi} imply 
\begin{eqnarray*}
&&\sum\limits_{n=1}^{N_\ve} \nabla\cdot \left(\ve (D_{x_n^\ve}  F^{n}_1)\nabla\phi_{\Omega_n^\ve} \right)  \to 0
\text{ in  } (H^{1}(\Omega))^\prime, \\
&&
 \sum\limits_{n=1}^{N_\ve}  \ve  \nabla_x \cdot \big(D_{x_n^\ve}F^{n}_1\big) \nabla \phi_{\Omega_n^\ve} \to 0 \, \text{  in } \, L^2(\Omega).
\end{eqnarray*}
Then the last convergences  together with  \eqref{H1_Conver22} result in
\begin{eqnarray}\label{H_1conv_1}
 \sum\limits_{n=1}^{N_\ve} \Big(\Big[ \tilde\Psi\Big(x,\frac {D_{x_n^\ve}^{-1} x} \ve\Big) - \ddashinttt_{Y_x}\Psi dy\Big] \nabla \phi_{\Omega_n^\ve} 
 +   \ve D_{x_n^\ve}   F_1^{n}  :\nabla^2 \phi_{\Omega_n^\ve} \Big)\to 0   \text{ in }   \big(H^{1}(\Omega)\big)^\prime. \qquad 
\end{eqnarray}
If $\rho< \frac 12$, due to  \eqref{estim_grad_phi},  convergence \eqref{convergence_H_1} follows    from \eqref{H_1conv_1}.  If $\frac 12\leq \rho <1$ we have to iterate the above calculations. The  periodicity of $h_1(x,\tilde y)$ yields   $\ddashinttt_{ Y} F_1(x, \tilde y) d\tilde y =\bf 0$. Thus  there exists a unique $h_2 \in W^{1,\infty}_0(\Omega; C^2_{\text{per}}(Y)^{d\times d})$ with zero $Y$-mean value that
\begin{equation*}
\Delta_{\tilde y} h_2(x,\tilde y)= F_1(x,\tilde y). 
\end{equation*}
Defining  $F_2(x,\tilde y)= \nabla_{\tilde y}  h_2(x, \tilde y)$  for $x\in \Omega$, $\tilde y \in Y$, 
and $F^{n}_2(x)=  F_2(x, D^{-1}_{x_n^\ve} x/\ve)$  for $x \in \Omega_n^\ve\cap \Omega$, we obtain a.e. in $\Omega$
\begin{eqnarray*}
 &&\sum\limits_{n=1}^{N_\ve}  \ve D_{x_n^\ve}   F_1^{n} : \nabla^2 \phi_{\Omega_n^\ve} =
\sum\limits_{n=1}^{N_\ve} \ve^2 \nabla\cdot\Big(D_{x_n^\ve}(D_{x_n^\ve} F_2^{n}) : \nabla^2 \phi_{\Omega_n^\ve} \Big) \\
&& - \sum\limits_{n=1}^{N_\ve}  \ve^2  \nabla_x \cdot\big(D_{x_n^\ve} (D_{x_n^\ve}F_2^{n}) \big) : \nabla^2 \phi_{\Omega_n^\ve}
 -\sum\limits_{n=1}^{N_\ve}  \ve^2 D_{x_n^\ve}(D_{x_n^\ve} F_2^{n}) : \nabla^3 \phi_{\Omega_n^\ve}  .
\end{eqnarray*}
Due to \eqref{estim_grad_phi},   there exists 
$m\in \mathbb N$  that $m(1-\rho)>1$ and  $\ve^{m-1}||\nabla^m \phi_{\Omega^\ve_n}||_{L^\infty(\mathbb R^d)}   \to 0$ as $\ve\to 0$.
Reiterating  the last  calculations $m$-times    yields   convergence \eqref{convergence_H_1}.

To show \eqref{convergence_H_1_2} we consider  \eqref{Eq_h} with  $x=x_n^\ve$  and  $\tilde y=D^{-1}_{x_n^\ve} x/\ve$    for  $x\in \Omega_n^\ve\cap \Omega$ and  $1\leq n \leq N_\ve$.  Then  $\tilde \Psi(x_n^\ve, D_{x_n^\ve}^{-1} x/\ve) = \Psi(x_n^\ve, x/\ve)$ and  for  $x\in \Omega_n^\ve\cap \Omega$ we have 
\begin{eqnarray*}\label{H1_sec1_2}
\Psi\Big(x_n^\ve, \frac{ x} \ve\Big)- \ddashinttt_{Y_{x_n^\ve}}\Psi(x_n^\ve, y) dy
&=& \ve \nabla\cdot\Big( D_{x_n^\ve}\cdot\nabla_{\tilde y}  h_1\Big(x_n^\ve, \frac{D^{-1}_{x_n^\ve} x}\ve\Big) \Big).
\end{eqnarray*}
Applying  similar calculations as  in the proof of \eqref{convergence_H_1} yields convergence \eqref{convergence_H_1_2}.
\hfill \end{proof}

The convergence in Lemma~\ref{H_1_conv} is now applied in the proof of the convergence result for a bounded in $H^1$ sequence, where   $\phi_{\Omega^\ve_n}$  will be the approximations of $\chi_{\Omega_n^\ve}$ with  $0<r<\rho<1$ and $n=1,\ldots, N_\ve$.  The proof emphasises the importance of $r<1$ in  locally-periodic approximations of functions with space-dependent periodicity. 

\begin{proof} ({\it Theorem \ref{compactness_theorem}} ) The ideas of the proof are similar to those  for the  two-scale convergence with periodic test functions,  \cite{Allaire, Lukkassen}. 
\\
Since     $\{u^\ve \}$ is bounded in $L^2(\Omega)$, thanks to   Theorem \ref{L2_lp_covregence}, there exist   $u_0 \in L^2(\Omega; L^2(Y_x))$ and  a subsequence,  denoted  again by $\{u^\ve\}$,  such that $u^\ve \to u_0 $ in locally-periodic two-scale sense.
 We shall  show that $u_0$ is independent of $y$. We consider    approximations $\phi_{\O_n^\ve}\in C^\infty_0 (\Omega_n^\ve)$  of    $\chi_{\Omega_n^\ve}$  satisfying properties \eqref{ApproxCharactF}.
The boundedness of   $\{\nabla u^\ve\}$ yields
\begin{equation*}
0=\lim\limits_{\ve\to 0} \int_\Omega \ve \nabla u^\ve \mathcal L^\ve\psi   \, dx
= \lim\limits_{\ve\to 0}\Big[ \int_\Omega \ve  \nabla u^\ve (\mathcal L^\ve\psi -\mathcal L^\ve_\rho\psi)\, dx -
 \int_\Omega \ve  u^\ve \nabla (\mathcal L^\ve_\rho\psi)\, dx\Big]
\end{equation*}
for $\psi \in W^{1, \infty}_0(\Omega; C_{\text{per}}^\infty(Y_x))$. Due to the convergence of $\phi_{\Omega_n^\ve}$ stated in  \eqref{ApproxCharactF} and boundedness of $\{\nabla u^\ve\}$ in $L^2(\Omega)$, the limit  of the first integral on the right hand side is equal to zero. The   second integral    can be written as the sum of three 
\begin{eqnarray*}
 \mathcal I_1+\mathcal I_2+\mathcal I_3= && \int_\Omega   u^\ve \, \mathcal L^\ve \nabla_y  \psi \, dx  
+ \int_\Omega \sum\limits_{n=1}^{N_\ve}
u^\ve D_{x_n^\ve}^{-T} \nabla_{\tilde y} \tilde \psi \Big(x, \frac {D_{x_n^\ve}^{-1} x} \ve\Big) 
\big( \phi_{\Omega_n^\ve}- \chi_{\Omega_n^\ve} \big)  dx \\
&&+   \ve\int_\Omega \sum\limits_{n=1}^{N_\ve}u^\ve  
\Big[\tilde \psi \Big(x, \frac {D_{x_n^\ve}^{-1} x} \ve\Big)  \nabla \phi_{\Omega_n^\ve}
+ \nabla_x \tilde \psi \Big(x, \frac {D_{x_n^\ve}^{-1} x} \ve\Big)   \phi_{\Omega_n^\ve} \Big] dx.
\end{eqnarray*}
The  properties  \eqref{ApproxCharactF}  of  $\phi_{\O_n^\ve}$   imply  that
$
\lim\limits_{\ve\to 0} \mathcal I_2=0$  and $ \lim\limits_{\ve\to 0}  \mathcal I_3=0. $
Considering  the locally-periodic two-scale convergence of   $\{u^\ve\}$ in  $\mathcal I_1$ we obtain
\[
 \int_\Omega  \ddashinttt_{Y_x} u_0 \nabla_y  \psi(x, y) dy   dx=0
\]
for $\psi \in W^{1, \infty}_0(\Omega; C^\infty_{\text{per}}(Y_x))$ and can conclude  that  $u_0$ is independent of $y$, see \cite{Alt}. Since the average over $Y_x$ of $u_0$ is equal to $u$ and $u_0$ is independent of $y$, we deduce that for any subsequence the locally-periodic two-scale limit reduces  to the weak $L^2$-limit $u$. Thus the entire sequence $\{u^\ve\}$ converges to $u$ in the locally-periodic two-scale sense.  

Applying    Theorem~\ref{L2_lp_covregence}  to the bounded in  $L^2(\Omega)^d$  sequence $\{\nabla u^\ve\}$ yields the existence of  $\xi \in L^2(\Omega; L^2(Y_x)^d)$ and  of a subsequence, denoted again by  $\{\nabla u^\ve\}$,    that
\begin{eqnarray}\label{conver_lts_grad}
&& 
 \lim\limits_{\ve \to 0} \int_\Omega \nabla u^\ve(x)\,  \mathcal L^\ve \Psi(x) \, dx=
\int_\Omega  \ddashinttt_{Y_x}
\xi \, \Psi( x, y) dy   dx
\end{eqnarray}
for  $\Psi \in W^{1, \infty}_0(\Omega; C^\infty_{\text{per}}(Y_x)^d)$.
 Now we assume additionally    $\nabla_y \cdot \Psi(x,y)=0$ for $x\in \Omega$,  $y\in Y_x$ and  notice that 
 \begin{equation}\label{part_int}
 \nabla \cdot\tilde \Psi\Big(x_n^\ve, \frac {D^{-1}_{x_n^\ve} x} \ve\Big)=  \nabla \cdot \Psi\Big(x_n^\ve, \frac x \ve\Big) = 
 \frac 1 \ve \nabla_y \cdot \Psi\Big(x_n^\ve,  \frac x\ve\Big)=0.
 \end{equation}
We rewrite the integral on the left hand side of \eqref{conver_lts_grad} in the form
\begin{eqnarray*}
&&I=\int_\Omega \nabla u^\ve (\mathcal L^\ve \Psi -  \mathcal L^\ve_0 \Psi) dx+ \int_\Omega \nabla u^\ve \sum\limits_{n=1}^{N_\ve} \Big[
 \Psi\big(x_n^\ve, \frac x \ve\big) -\ddashinttt_{Y_{x_n^\ve}}\Psi(x_n^\ve, y) dy   \Big] \phi_{\Omega_n^\ve} dx \\
 &&+
  \int_\Omega \nabla u^\ve \sum\limits_{n=1}^{N_\ve} 
\Big[ \Psi\big(x_n^\ve, \frac x \ve\big) -\ddashinttt_{Y_{x_n^\ve}}\Psi(x_n^\ve,y) dy \Big]
(\chi_{\Omega_n^\ve}- \phi_{\Omega_n^\ve} ) dx  +\int_\Omega \nabla u^\ve \ddashinttt_{Y_x}\Psi(x, y)  dy dx
\\
&&+    \int_\Omega \nabla u^\ve \sum\limits_{n=1}^{N_\ve} \Big[\ddashinttt_{Y_{x_n^\ve}}\Psi(x_n^\ve, y) dy - \ddashinttt_{Y_x}\Psi(x, y)  dy\Big]\chi_{\Omega_n^\ve} dx= I_1+I_2+I_3+I_4+I_5 .
\end{eqnarray*} 
The  boundedness of $\{u^\ve\}$ in $H^1(\Omega)$ and  the continuity of $\Psi$ and $D$   ensure 
   $\lim\limits_{\ve \to 0} I_1=0$ and $\lim\limits_{\ve \to 0} I_5=0$.
Now we integrate by parts in $I_2$ and apply equality  \eqref{part_int}. Then 
the $H_1$-boundedness of $\{u^\ve \}$  yields  
\begin{eqnarray*}
|I_2|&=&\Big|\int_\Omega  u^\ve \sum\limits_{n=1}^{N_\ve} 
\Big[ \Psi\big(x_n^\ve, \frac x \ve\big)  -\ddashinttt_{Y_{x_n^\ve}}\Psi(x_n^\ve, y)  dy \Big] \nabla \phi_{\Omega_n^\ve}  dx\Big| \\
&& \qquad \leq C
\Big\| \sum\limits_{n=1}^{N_\ve} \Big[ \Psi\big(x_n^\ve, \frac x \ve\big)   -\ddashinttt_{Y_{x_n^\ve}}\Psi(x_n^\ve, y)  dy  
\Big] \nabla \phi_{\Omega_n^\ve}  \Big\|_{(H^{1}(\Omega))^\prime}.
\end{eqnarray*} 
Applying now convergence  \eqref{convergence_H_1_2} from  Lemma \ref{H_1_conv}  we obtain  $\lim\limits_{\ve \to 0} I_2=0$.
The  convergence  in  \eqref{ApproxCharactF},  the regularity of $\Psi$ and 
boundedness of $\{u^\ve \}$ in $H^1(\Omega)$ imply that  $\lim\limits_{\ve \to 0} I_3=0$.
  Finally,  the integration by parts  and  the $L^2$-convergence of  $\{u^\ve\}$  give 
\begin{equation*}
\lim\limits_{\ve \to 0} I_4 = -\int_\Omega  u(x) \nabla \cdot  \Big(\ddashinttt_{Y_x} \Psi(x, y) d y \Big) dx=
\int_\Omega  \,  \ddashinttt_{Y_x} \nabla u(x)  \Psi(x, y) d y dx.
\end{equation*}
Thus, for any $\Psi \in W^{1, \infty}_0(\Omega; C^\infty_{\text{per}}(Y_x)^d)$ with $\nabla_y\cdot \Psi(x,y) =0$,  we have
\begin{equation*}
 \int_\Omega \,  \ddashinttt_{Y_x}\big(\xi - \nabla u(x) \big) \Psi(x, y) d y dx =0.
\end{equation*}
The Helmholtz decomposition, \cite{Allaire, Galdi, Nguetseng},   yields  that the orthogonal to solenoidal fields are  gradient fields, i.e. 
there exists a function $u_1$ from $\Omega$ to $H^1_{\text{per}}(Y_x)/\mathbb R$, such that 
$ \xi(x, \cdot) - \nabla u(x) = \nabla_y u_1(x, \cdot) $ for a.a. $x \in \Omega$. Then, using the integrability of $\xi$ and $\nabla u$ we conclude that  
$u_1 \in L^2(\Omega; H^1_{\text{per}}(Y_x)/\mathbb R)$. 
\hfill \end{proof}

In analogue  to the two-scale convergence with periodic test functions, \cite{Allaire, Lukkassen},  we  show, under an additional assumption, the convergence of the product of two locally-periodic two-scale convergent sequences.
\begin{lemma}\label{StrongTwo_ScaleConvergence}
Let $\{u^\ve\}\subset L^2(\Omega)$ be a sequence   that converges locally-periodic two-scale  to $u\in L^2(\Omega; L^2(Y_x))$   and assume that 
\begin{equation}\label{strong_two_scale}
\lim\limits_{\ve\to 0}\| u^\ve \|^2_{L^2(\Omega)}=\int_\Omega   \ddashinttt_{Y_x}  |u(x,y)|^2   dy dx.
\end{equation}
Then for  $\{v^\ve\}\subset L^2(\Omega)$ that  converges  locally-periodic two-scale  to $v\in L^2(\Omega; L^2(Y_x))$ we have
\[
u^\ve(x) v^\ve(x) \rightharpoonup \ddashinttt_{Y_x} u(x,y) v(x,y) dy \quad  \text{ weakly in } \mathcal D^\prime(\Omega) \quad \text{ as }\,  \ve \to 0.
\]
\end{lemma}
\begin{proof}
Let $\{\psi_k\}$ be a sequence of  functions in $L^2(\Omega; C_{\text{per}}(Y_x))$ that converges to $u$ in  $L^2(\Omega; L^2(Y_x))$. 
Convergence \eqref{LocPeriodConver} for $L^2(\Omega; C_{\text{per}}(Y_x))$-functions,
the definition of the  locally-periodic two-scale convergence, and  assumption \eqref{strong_two_scale} ensure 
\begin{eqnarray*}
\lim\limits_{\ve\to 0}\int_\Omega \left[ u^\ve(x) -\mathcal L^\ve \psi_k (x)  \right]^2 dx =
\int_\Omega \ddashinttt_{Y_x} [u(x,y)- \psi_k(x,y) ]^2 dy dx. 
\end{eqnarray*}
The limit  as $k\to \infty$ in  the last equality and 
the  strong convergence of $\psi_k$ to $u$ imply
\begin{equation}\label{ConverStrogAppr}
\lim\limits_{k\to \infty}\lim\limits_{\ve\to 0}\int_\Omega \left[ u^\ve(x) - \mathcal L^\ve  \psi_k(x) \right]^2 dx =0.
\end{equation}
For $\phi\in \mathcal D (\Omega)$ we consider now
\begin{eqnarray*}
\int_\Omega  u^\ve(x) v^\ve(x) \phi(x) dx = \int_\Omega  \mathcal L^\ve  \psi_k(x)  v^\ve (x) \phi(x) dx +
\int_\Omega \left[ u^\ve(x)- \mathcal L^\ve \psi_k(x) \right]  v^\ve (x) \phi(x)  dx.
\end{eqnarray*}
Applying   the l-t-s convergence and $L^2$-boundedness  of $v^\ve$  in the last equality we obtain
\begin{eqnarray*}
\Big| \lim\limits_{\ve \to 0} \int_\Omega  u^\ve(x) v^\ve(x) \phi   \, dx-
\int_\Omega \ddashinttt_{Y_x}   \psi_k(x,y)  v(x,y)  \phi \, dy dx \Big| 
\leq C  \lim\limits_{\ve\to 0}\int_\Omega \left| u^\ve -\mathcal L^\ve \psi_k\right|^2 dx.
\end{eqnarray*}
Then, letting  $k\to \infty$ and using \eqref{ConverStrogAppr} we obtain  the convergence stated in Lemma. 
\hfill\end{proof}

We shall refer  to    l-t-s convergent sequence  satisfying \eqref{strong_two_scale} as strongly  l-t-s convergent. 


\section{Homogenization of  plywood structures}\label{Application}
Now we return to our main  problem \eqref{main}, presented in  Section~\ref{ModelD}. Results in this section require us to introduce some standard   regularity and ellipticity constraints on the given vector functions $G$, $g$ and  tensors $E_1$, $E_2$. We assume $g \in H^1(\Omega)$, $G\in L^2(\Omega)$ and   $E_1, E_2$ are symmetric, i.e. $E_{m, ijkl}=E_{m, klij}=E_{m, jikl}=E_{m, ijlk}$ for $m=1,2$, and positive definite, i.e  $E_1 \xi: \xi \geq \alpha |\xi|^2$, $E_2 \xi: \xi \geq \beta |\xi|^2$  for all symmetric matrices $\xi \in \mathbb R^{3\times 3}$ and   $\alpha,  \beta>0$.

\begin{definition}
 The function $u^\ve$ is called a weak solution of the problem~\eqref{main} if $u^\ve - g \in H^1_0(\Omega) $   and  satisfies 
\begin{equation}\label{main1}
 \int_\Omega E^\ve(x) \, e(u^\ve)\, e(\psi)\, dx= \int_\Omega G (x)\, \psi\, dx \, \hfill\qquad  \text{ for all } \, \,  \psi \in H^1_0(\O).
\end{equation} 
\end{definition}

Due to our assumptions, the tensor $E^\ve$, given by  \eqref{stif_loc_period}, satisfies the Legendre conditions. Since $E_1$ and $E_2$ are constant, we have also the uniform boundedness  of $E^\ve$.  Thus, there exists a unique weak solution of the problem \eqref{main},  see \cite{Allaire2002}, and 
\begin{lemma}\label{apriori} 
Any solution of   the  model \eqref{main}  satisfies the  estimate 
 \begin{eqnarray*}
  ||u^\ve||_{H^1(\Omega)} \leq C,
 \end{eqnarray*}
where  $C$ is a constant independent of $\ve$.
\end{lemma}

{\it Proof sketch}. Considering $u^\ve - g$ as a test function in \eqref{main1} and applying the regularity assumptions on $G$ and $g$,  the coercivity of $E^\ve$, and the Korn inequality for functions in  $H^1_0(\Omega)$, we obtain the stated estimate. \qed

To define the effective elastic properties of a material with plywood microstructure we shall derive macroscopic equations for the microscopic model \eqref{main} using the notion of the locally-periodic  two-scale convergence,  introduced in Section~\ref{TwoScaleConv}.  

In the    locally-periodic plywood structure  the rotation angle is constant in each layer $L_k^\ve$ and the characteristic function of the domain occupied by fibres can  be defined considering an additional  division of  $L_k^\ve$ into  cubes  of side $\ve^r$ with $r\in (0,1)$. Notice that the microstructure is give by the rotation around the fix $x_3$-axis with a rotation angle dependent only on $x_3$.
Thus  the rotation of a correspondent set of cubes will  reproduce  the cylindrical structure of the fibres.
Regarding a possible change of enumeration, we consider  the layers $L_k^\ve$  such that $\Omega\subset \cup_{k=1}^{K_\ve} \overline L_k^\ve$ and  $L_k^\ve \cap \Omega \neq \emptyset$. Then  the number of layers having non-empty intersection with  $\Omega$ satisfies  $K_\ve\leq(\text{diam}(\Omega)+2)\ve^{- r}$. Each layer $L_k^\ve$ we can divide into open non-intersecting cubes $\Omega_n^\ve$ of  side $\ve^r$  and consider a family of  cubes $\{\Omega_n^\ve\}_{n=N_\ve^{k-1}+1}^{N_\ve^k}$ such that 
 \begin{equation*}
  \Omega \cap L_k^\ve \subset \cup_{n=N_\ve^{k-1}+1}^{N_{\ve}^k} \overline\Omega_n^\ve \, \text{ and } 
  \, \Omega_n^\ve \cap (\Omega\cap L_k^\ve)\neq \emptyset, 
\end{equation*}
where  $N_\ve^0=0$ and  $N_\ve^k \leq (\text{diam}(\Omega) +2)^2\ve^{-2 r}$ for $k=1, \ldots, K_\ve$. In this way we obtain a covering of $\Omega$ by the family of  cubes satisfying the  estimates on $N_\ve=\sum_{k=1}^{K_\ve} N_\ve^k$, $\tilde N_\ve$, and $\mathcal K_\ve$,  stated in Section~\ref{TwoScaleConv}.
Then  for $k=1,\ldots, K_\ve$ and for any $n, m \in \mathbb N$ with $N_\ve^{k-1}+1\leq n,m \leq N_\ve^k$ we  choose $x_n^\ve \in \Omega_n^\ve$  and $x_m^\ve \in \Omega_m^\ve$ such that $x_{n,3}^\ve= x_{m,3}^\ve$, i.e. the points $x_n^\ve$ have the same third component if they belong to the same layer $L_k^\ve$.

For $x\in \mathbb R^3$ we consider  now the deformation matrix  given by the rotation matrix $R^{-1}_{x_3}=R^{-1}(\gamma(x_3))$, with
  $R$  and  $\gamma$  defined  in Section~\ref{ModelD},  and obtain 
a continuous family of rotated  cubes $Y_{x_3}=R^{-1}_{x_3}Y$. 
Due to regularity of $\gamma$, i.e. $\gamma \in C^2(\mathbb R)$, and $|\det R^{-1}_{x_3}|=1$ 
for all $x\in \mathbb R^3$,  the  matrix  $R^{-1}_{x_3}$ fulfils assumptions posed in Section~\ref{TwoScaleConv}.

Using the notion of the locally-periodic approximation introduced in Section~\ref{TwoScaleConv},  the characteristic function  $\chi_{\Omega_f^\ve}$,  defined in  \eqref{charac_local_per}, can be written as 
\begin{equation*}
\chi_{\Omega_f^\ve}(x)= \mathcal L^\ve_0 \eta(x)\, \chi_\Omega(x),
\end{equation*}
where $\eta(x, y)= \tilde \eta(R_{x_3} y)$ for $y\in Y_{x_3}$ and $\tilde \eta $  given by  \eqref{char}. Here we choose    $\tilde x_n^\ve= R_{x_n^\ve}^{-1} \ve k$ for some $k\in \mathbb Z^3$ and $R_{x_n^\ve}^{-1}=R^{-1}(\gamma(x_{n,3}^\ve))$. The  function  $\tilde \eta $ is constant with respect to $\tilde y_1$ and we  can  introduced formally the periodicity with respect to the first variable. 

Now,   each $\Omega^\ve_{n}$ can be covered by a family of closed  cubes
$\Omega^\ve_{n}\subset \cup_{i=1}^{I_n^\ve} \ve  Y^i_{x_n^\ve}$ such that 
 $\ve Y^i_{x_n^\ve}\cap\overline \Omega^\ve_{n}\neq \emptyset$, where $Y^i_{x_n^\ve}= R_{x_n^\ve}^{-1} (Y+k_i)$ with $k_i \in \mathbb Z^3$. The number $I_n^\ve$, the subset   $\mathcal M_\ve^n=\Omega^\ve_{n}\setminus \cup_{i=1}^{\tilde I_n^\ve} \ve  Y^i_{x_n^\ve}$,  and 
 the number  $\tilde I_n^\ve$ of all cubes enclosed in  $\Omega^\ve_{n}$ 
satisfy the estimates stated in Lemma~\ref{converg_ost} in  Section~\ref{TwoScaleConv}.

We consider  the sequence  $\{u^\ve\}$ of solutions of  \eqref{main}.  A priori estimate in Lemma~\ref{apriori} 
 ensures   the existence of    $u \in H^1(\Omega)$ and  of a  subsequences, denoted again by $\{ u^\ve \}$, such that $u^\ve \rightharpoonup u$ in $H^1$.  Thanks to Theorem~\ref{compactness_theorem}  there exists  another  subsequences of $\{\nabla u^\ve\}$,
 denoted again by $\{\nabla u^\ve \}$, and
  $u_1 \in L^2(\Omega; H^1_{\text{per}}(Y_{x_3})/\mathbb R)$ such that  
\[
u^\ve \to u \quad  \text{ and } \quad \nabla u^\ve \to \nabla u + \nabla_y u_1  
\]
in the  locally-periodic two-scale sense. We consider 
\begin{equation*}
\psi= \psi_1(x) + \ve (\mathcal L^\ve_\rho \psi_2)(x)
\end{equation*}
as a test function in \eqref{main1}, where $ \psi_1\in C^\infty_0(\Omega)$ and $\psi_2 \in W^{1,\infty}_0(\Omega; C^\infty_{\text{per}}(Y_{x_3})). $ 
The regularity assumptions  on $\psi_1$, $\psi_2$ and $ \phi_{\O_n^\ve}$ ensure that  $\psi\in H^1_0(\Omega)$. Then \eqref{main1} reads
\begin{equation}\label{main2}
\int_\Omega (\mathcal L^\ve_0 A)(x)  e(u^\ve) \left(e(\psi_1) + \ve   e(\mathcal L^\ve_\rho\psi_2) \right)dx=
 \int_\Omega G(x) \left(\psi_1 + \ve  \mathcal L^\ve_\rho\psi_2 \right)dx,  
\end{equation}
where $ A(x,y)=E_1\eta(x,y) +E_2(1-\eta(x,y))$. Notice that  $ A\in L^\infty(\cup_{x\in \Omega}\{x\}\times Y_{x_3})$ and $A\in C(\overline\Omega; L^p_{\text{per}}(Y_{x_3}))$ for $1\leq p <\infty$. Since the dependence on $x$ in $A$ occurs only due to its $Y_{x_3}$-periodicity, we have that
  $\mathcal L^\ve_0 A(x)=\mathcal L^\ve A(x) $   a.e. in $\Omega$.
  
To apply the locally-periodic  two-scale convergence in \eqref{main2}  we have to bring the test function in the form involving the locally-periodic approximation,  i.e. replace $\phi_{\Omega_n^\ve}$  by $\chi_{\Omega_n^\ve}$. We rewrite  the left hand side  of \eqref{main2} in the form
\begin{eqnarray}\label{RewriteA}
&& 
\int_\Omega (\mathcal L^\ve_0 A)\,  e(u^\ve) 
\Big[e(\psi_1)+ 
\ve \sum\limits_{n=1}^{N_\ve}\left(  e(\psi_2^{n}) \chi_{\Omega_n^\ve}+  
e(\psi_2^{n})\left(\phi_{\O_n^\ve} -  \chi_{\Omega_n^\ve} \right)
+   \psi_2^{n} \odot \nabla \phi_{\Omega_n^\ve}\right) \Big] dx,\nonumber
\end{eqnarray}
 where $\psi_2^n(x)=\tilde\psi_2 \left(x,{R_{x_n^\ve} x/\ve}\right)$ and $a\odot b= \big(\frac 12(a_i b_j+ a_j b_i)\big)_{1\leq i, j\leq 3}$  for $a, b \in \mathbb R^3$.
 
Considering $\ve \nabla \psi_2^{n}(x)=\ve\nabla_x \psi_2^{n}+ R^{T}_{x_n^{\ve}}\nabla_{\tilde y} \psi_2^n$ and using the regularity of $\psi_2$  together with convergences in Lemma~\ref{converg_ost} applied to $\mathcal L^\ve e_y(\psi_2)$    we obtain
  \begin{eqnarray*}
\lim\limits_{\ve\to 0}\int_{\Omega}  \sum\limits_{n=1}^{ N_\ve}  |\ve e(\psi_2^{n})|^p \chi_{\Omega_n^\ve}=
\int_{\Omega} \ddashinttt_{Y_{x_3}} |e_y(\psi_2(x,y))|^p dy dx  \quad \text{ for } 1\leq p <\infty.
\end{eqnarray*}
Convergence \eqref{EstimOscil2} in Lemma~\ref{converg_ost} implies   
  \begin{eqnarray*}
\lim\limits_{\ve\to 0}\int_{\Omega}|\mathcal L^\ve_0 A(x)|^p dx  =\int_{\Omega} \ddashinttt_{Y_{x_3}}|A(x,y)|^p dy dx   \quad \text{ for } 1\leq p <\infty.
\end{eqnarray*}
Then,   for a sequence $\{\Phi^\ve\}\subset L^\infty(\Omega)$ given by
$$\Phi^\ve(x)=\mathcal L^\ve_0 A(x) \, 
\Big[e(\psi_1(x))+ \ve \sum\limits_{n=1}^{N_\ve}  e(\psi_2^{n}(x)) \chi_{\Omega_n^\ve}(x) \Big]
$$ 
we can conclude that  $\{\Phi^\ve\}$ is  bounded in  $L^2(\Omega)$ and satisfies   assumption \eqref{strong_two_scale} of the strong locally-periodic  two-scale convergence  stated in Lemma~\ref{StrongTwo_ScaleConvergence}, i.e.  
\begin{eqnarray*}
 \lim\limits_{\ve \to 0}\int_{\Omega}  |\Phi^\ve(x)|^2  dx
=\int_{\Omega}\ddashinttt_{Y_{x_3}} \left|A(x,y)\big[ e(\psi_1(x)) + e_y(\psi_2(x,y))\big]\right|^2 dydx.
\end{eqnarray*}
Boundedness of $\{u^\ve\}$ in $H^1(\Omega)$, $L^2$-convergence of  $\phi_{\O_n^\ve}$,  and   the regularity of $\psi_2$ give
\begin{eqnarray}\label{LHS_A}
 && \lim\limits_{\ve\to 0}  \ve \int_{\Omega }  \mathcal L^\ve_0 A \,   e(u^\ve)   
\sum\limits_{n=1}^{N_\ve} e(\psi_2^{n}) (\phi_{\O_n^\ve} 
- \chi_{\Omega_n^\ve}) dx =0.
\end{eqnarray}
 The assumption  $||\nabla \phi_{\O_n^\ve}||_{L^\infty(\mathbb R^d)}\leq C \ve^{-\rho}$, where
$0<r<\rho<1$, ensures  
\begin{equation}\label{LHS_Grad}
 \lim\limits_{\ve\to 0} \ve \int_{\Omega } \mathcal L^\ve_0  A \, e(u^\ve)  \sum\limits_{n=1}^{N_\ve}
\psi_2^{n} \odot\nabla \phi_{\O_n^\ve}
 dx=0.
\end{equation}
The regularity  of $\psi_2$ and  $G$ imply that the second term on the right hand side of \eqref{main2} convergences to zero as $\ve\to 0$. 

 Thus, taking into account  convergences  \eqref{LHS_A} and \eqref{LHS_Grad},  the strong l-t-s convergence of $\{\Phi^\ve\}$,  the l-t-s convergence for a  subsequence of $\{\nabla u^\ve\}$, denoted again by $\{\nabla u^\ve\}$,  we can pass to the limit as $\ve \to 0$ in 
\eqref{main2} 
and obtain    
\begin{equation*}
 \int_\Omega  \ddashinttt_{Y_{x_3}} A(x,y) \big(e(u)+   e_y(u_1(x,y)) \big)\big(  e(\psi_1)+ e_y(\psi_2(x,  y))\big) dy dx= \int_\Omega   G(x) \psi_1  dx.
\end{equation*}
Then   coordinate transformation $\mathcal F : Y_{x_3}\to  Y$, i.e. $\tilde y = \mathcal F(y)=  R_{x_3}y$, yields
\begin{equation}\label{macro1}
 \int_\Omega  \ddashinttt_{ Y }\tilde A(\tilde y)\big(e(u)+ e^R_{\tilde y}(\tilde u_1)\big)
\big(e(\psi_1)+ e^R_{\tilde y}(\tilde \psi_2) \big)d\tilde y dx=  \int_\Omega  G(x)\,  \psi_1  dx,
\end{equation}
where 
\begin{equation}\label{def_trans_grad} 
e^R_{\tilde y,kl}(v)=\frac 12 \Big(\big[{R}^{T}_{x_3} \nabla_{\tilde y}  v^l\big]_k 
+\big[R^{T}_{x_3} \nabla_{\tilde y}  v^k\big]_l \Big)
\end{equation}
and $\tilde A(\tilde y)=E_1 \tilde\eta(\tilde y)+ E_2(1-\tilde\eta(\tilde y))$ for $\tilde y \in Y$. 

By density argument,  \eqref{macro1} holds also for 
$\psi_1 \in H^1_0(\Omega)$ and $\tilde \psi_2 \in L^2(\Omega; H^1_{\text{per}}(Y)/\mathbb R)$. 
Taking $\psi_1=0$ and using  linearity of the problem we  conclude that  $\tilde u_1$  has the form
\[
\tilde u_1(x,\tilde y)= \frac 12\sum\limits_{i,j=1}^3 \Big( \frac{\partial u^i(x)}{\partial x_j} + \frac{\partial u^j(x)}{\partial x_i} \Big)
\tilde \omega_{ij}(x_3,\tilde y),
\]
where $\tilde \omega_{ij}(x_3,\tilde y)$ are solutions of  unit cell problems  
\begin{eqnarray}\label{unit_loc_per}
-\nabla_{\tilde y}\cdot \Big(
  {R}_{x_3} \tilde A (\tilde y)  e^R_{\tilde y}(\tilde \omega_{ij}) \Big)
  =  \nabla_{\tilde y}\cdot \Big({R}_{x_3} \tilde A(\tilde y) l_{ij} \Big) & \quad \text{ in }  Y, \\
 \tilde{\omega}_{ij} \hspace{1.0 cm} \text{ periodic } & \quad \text{ in }  Y. \nonumber
\end{eqnarray}
Here ${l}_{ij}= \frac 12 (l_i \otimes l_j + l_j \otimes l_i)$ are symmetric matrices, whereas  $(l_i)_{1\leq i\leq 3}$ is the canonical basis of $\mathbb R^{3}$.
 Since $\tilde A$ is independent of $\tilde y_1$ and solutions of the problems \eqref{unit_loc_per} are unique up to a constant, 
 we obtain that $\tilde\omega_{ij}$ does not depend on $\tilde y_1$. Thus  \eqref{unit_loc_per} can be reduced to the two-dimensional problems 
 \begin{eqnarray}\label{unit_loc_per2}
-\nabla_{\hat y}\cdot \Big(
 \hat {R}_{x_3} \tilde A(\hat y) \hat e_{\hat y}^R(\tilde{\omega}_{ij})\Big) 
  = \nabla_{\hat y}\cdot \Big(\hat {R}_{x_3} \tilde A(\hat y) l_{ij} \Big) & \quad \text{ in } \hat  Y, \\
 \tilde{\omega}_{ij} \hspace{1.0 cm} \text{ periodic } & \quad \text{ in }  \hat Y, \nonumber
\end{eqnarray}
 where $\hat Y= Y \cap \{\tilde y_1=0\}$ with  $\hat y=(\tilde y_2, \tilde y_3)$,  $\tilde A(\hat y):=\tilde A(\tilde y)$, and 
 \begin{equation}\label{transform_reduce}
\hat e_{\hat y,kl}^R(v)=\frac 12 \Big[\big(\hat {R}^{T}_{x_3} 
\nabla_{\hat y} v^l\big)_k +\big(\hat {R}^{T}_{x_3} \nabla_{\hat y} v^k\big)_l \Big], \quad
\hat{R}_{x_3}\hspace{-0.1 cm}=\hspace{-0.1 cm} \left(\hspace{-0.15 cm}
\begin{array}{ccc}
-\sin \gamma(x_3) & \hspace{-0.1 cm} 
 \cos \gamma(x_3)& \hspace{-0.15 cm} 0 \\
 0&\hspace{-0.1 cm}  0&\hspace{-0.15 cm} 1
\end{array}\hspace{-0.1 cm}\right).
\end{equation}
Thus,  for locally-periodic plywood structure  we have that 
\begin{theorem}\label{Theorem_macro_loc_per}
 The sequence of microscopic solutions $\{u^\ve\}$ of \eqref{main}, with the elasticity tensor given by \eqref{stif_loc_period}, converges to a  solution of the macroscopic problem 
\begin{equation}\label{macro-local}
\begin{cases}
 -\text{div }  \sigma(x_3, u) = G & \quad \text{ in } \Omega, \\
\hspace{1.9 cm} u=g  & \quad \text{ on } \partial\Omega,
\end{cases}
\end{equation}
where $\sigma(x_3,u)= A^{\text{hom}}(x_3) e(u)$ with
$A^{\text{hom}} $  given by
\begin{eqnarray*}
A^{\text{hom}}_{ijkl}(x_3) &=&\ddashinttt_{\hat Y} \Big(\tilde A_{ijkl}(\hat y) +
\tilde A(\hat y)\hat e_{\hat y}^R (\tilde \omega_{ij})_{kl} \Big)d\hat y
\end{eqnarray*}
and  $\tilde \omega_{ij}$  are  solutions of the cell problems \eqref{unit_loc_per2}.
\end{theorem}

Considering the properties of the  matrix $R(\gamma(x_3))$ and the fact that $E_1$ and $E_2$ are  constant, symmetric and positive definite, yields  that  the homogenized tensor $A^{\text{hom}}$ is symmetric, positive definite and uniformly bounded.  This ensures  the existence of a unique weak  solution of the macroscopic model \eqref{macro-local}  and the convergence of the entire sequence of solutions of the microscopic problems \eqref{main}. 


\begin{figure}
\begin{center}
{\includegraphics[width=0.85\textwidth]{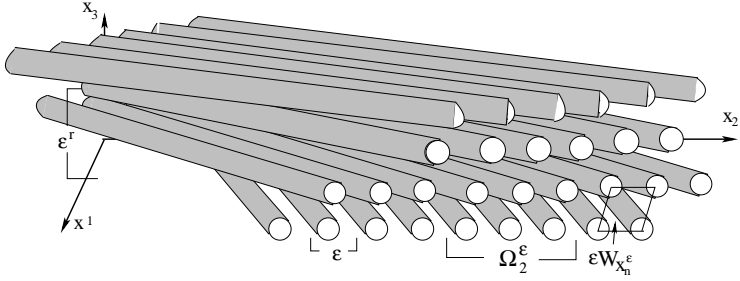}
\caption{\small{Non-periodic plywood-structure}} \label{fig_non_per_plywood}}
\end{center}
\end{figure} 

Now we consider the non-periodic plywood structure, where the layers of fibres aligned in the same direction are of the height $\ve$. For the analysis of the non-periodic problem it is convenient to define  the characteristic function of the domain  occupied by fibres in a different form,   equivalent to  \eqref{charac_local_per} for $r=1$.
 We consider the  function
 \[\vartheta(y)= \begin{cases}       1, \quad |\hat y|\leq a,\\         0, \quad |\hat y|>a       \end{cases}
\]
with $\hat y = (y_2, y_3)$ and  $a< 1/2$. 
For $k\in \mathbb Z^3$ we define   $ x_{k}^\ve=R^{-1}_{k} \ve k$ with $R^{-1}_{k}:= R^{-1}(\gamma(\ve k_3))$. Notice  that  $x_{k,3}^{\ve} =\ve k_3$,    the third variable is invariant under the rotation $R^{-1}_k$. Then the characteristic function of fibres for non-periodic microstructure reads
\begin{equation}\label{elastnonper}
\chi_{\Omega_f^\ve} (x)
=\chi_{\Omega}(x)\sum\limits_{k\in \mathbb Z^3}  \vartheta_\ve\big(R_{k}( x- x_{k}^\ve)\big), \quad \text{ where } \quad  \vartheta_\ve(x)= \vartheta\left(\frac x\varepsilon\right).
\end{equation}

To derive the macroscopic equations for the model \eqref{main} with elasticity tensor $E^\ve=E_1\chi_{\Omega_f^\ve}+ E_2(1-\chi_{\Omega_f^\ve})$,  where $\chi_{\Omega_f^\ve}$ is given by \eqref{elastnonper}, we shall approximate it by a  locally-periodic  problem and apply the locally-periodic two-scale convergence. The following calculations illustrate the motivation for the  locally-periodic approximation. 
 
 We consider  a partition covering of  $\Omega \subset \cup_{n=1}^{N_\ve}\overline \Omega_n^\ve$, as defined  in Section~\ref{TwoScaleConv}. For  $n=1,\ldots, N_\ve$  we choose  $\kappa_n \in \mathbb Z^3$  such that for $x_{n}^\ve=R^{-1}_{\kappa_{n}}\ve \kappa_n $ we have $x_{n}^\ve \in \Omega_n^\ve$.  We  cover $\Omega_n^\ve$ by shifted parallelepipeds  $\Omega_n^\ve\subset x_{n}^\ve +\cup_{j=1}^{I_n^\ve} \ve Y^j_{x_{n}^\ve}$, where    $Y^j_{x_{n}^\ve}=D_{x_{n}^\ve}(Y+m_j)$ for  $m_j\in \mathbb Z^3$ and a matrix $D(x)$, that will be specified later. Then    for    $1\leq j \leq I_n^\ve$  we consider   $k_j^n = \kappa_n + m_j$ and   $x_{k_j^n}^\ve= R^{-1}_{k_{j}^{n}}\ve k_j^n$.  Using the  regularity assumptions on $\gamma$  and  the Taylor expansion for $R$ around  $x_{\kappa_n}^\ve$, i.e. around $\ve\kappa_{n,3}$,  we obtain 
\begin{eqnarray}\label{Taylor_Approx}
 &&R_{k_{j}^{n}}(x- x_{k_j^n}^\ve) =R_{k_{j}^{n}} x- \ve k_j^n \nonumber \\
 &&\qquad \quad  = R_{\kappa_{n}}x +  R_{\kappa_n}^\prime x_{n}^\ve  m_{j,3}\ve +R_{\kappa_n}^\prime(x- x_{n}^\ve) 
 m_{j,3}\ve+ b(|m_{j,3} \ve|^2)x -  \ve(\kappa_n+ m_j) \nonumber\\ &&\qquad \quad
= R_{\kappa_{n}}(x  -x_{n}^\ve)- \tilde W_{x_{n}^\ve}m_j\ve +R_{\kappa_{n}}^\prime(x- x_{n}^\ve)  m_{j,3}\ve+ b(|m_{j,3} \ve|^2)x,  
\end{eqnarray}
where $\tilde W_{x_{n}^\ve}= \tilde W(x_{n}^\ve)$ with $\tilde W(x)=(I- \nabla R(\gamma(x_3)) x)$.  The notation of the gradient is understood as  $\nabla R(\gamma(x))x=\nabla_z(R(\gamma(z))x)|_{z=x}$. Thus for  $x\in \Omega_n^\ve$ the distance  between  
  $R_{\kappa_n} (x -x_n^\ve)- \tilde W_{x_n^\ve} m_j\ve$ and   $R_{k_{j}^{n}}(x-  x_{k_j^n}^\ve)$ is of  the order $\sup\limits_{1\leq j\leq I_n^\ve} |m_j\ve|^2 \sim \ve^{2r}$.
This will  assure that the non-periodic plywood structure can by approximated by locally-periodic,  comprising  $Y_{x_n^\ve}$-periodic structure  in each $\Omega_n^\ve$ of side $\ve^r$ with an appropriately  chosen  $r\in (0,1)$.  Here   $Y_x= D(x)Y$  with   $D(x)=R^{-1}(\gamma(x_3))W(x)$ and 
\[
W(x) =  \left(\hspace{-0.1 cm}
\begin{array}{ccc}
1 & 0 & 0\\
 0&1 & w(x)\\
 0&0&1
\end{array}\hspace{-0.15 cm}\right), \quad \text{ where } \, \, w(x)=  \gamma^\prime(x_3)(\cos (\gamma(x_3))x_1+\sin(\gamma(x_3)) x_2).
\]
The definition of   $R$, $W$ and $\gamma$  ensures   assumptions on  $D$  stated in Section~\ref{TwoScaleConv}. 
Since $\vartheta$ is independent of the first variable,   we consider  in $W(x)$ the  shift only for the second variable.
We denote $Z_x=W_xY$ and  consider a $Z_x$-periodic function 
\[
\hat \vartheta(x,y)= \sum_{k\in \mathbb Z^3} \vartheta(y- W_x k),
\]
Notice that  $\{y \in \mathbb R^3 : \hat \vartheta(x,y)=1\}$ is  a $Z_x$-periodic set of cylinders of radius $a$. 

Then for  $x_{n}^\ve \in \Omega_n^\ve$  as above and $\tilde x_n^\ve =x_{n}^\ve$ we define a   tensor 
\begin{eqnarray*}
  E^\ve_{n}(x) &=& (\mathcal L^\ve_0 B)(x)\, \chi_\Omega(x),  
\end{eqnarray*}
where  $B(x,y)=  E_1\hat \vartheta(x,R_{x_3}y) +E_2(1- \hat \vartheta(x,R_{x_3}y))$  for $x\in \Omega$ and  $y\in Y_x$.
 Notice that   $B\in L^\infty(\cup_{x\in \Omega}  \{x\}\times Y_x)$ and $B \in C(\overline\Omega, L^p_{\text{ per}}(Y_x))$ for $1\leq p < \infty$. 

Now we  rewrite the equation \eqref{main1} in the form 
\begin{eqnarray}\label{Aprox_Eq}
\int_\Omega E^\ve_{n}(x) e(u^\ve) e (\phi) dx + \int_\Omega(E^\ve(x)- E^\ve_{n}(x)) e(u^\ve) e (\phi) dx
 = \int_\Omega G(x)\,  \phi \,  dx
\end{eqnarray}
and shall show that the second integral on the left hand side converge to zero as $\ve \to 0$. Applying l-t-s convergence in the first term we shall obtain  macroscopic equations for the linear elasticity problem posed in a domain with a non-periodic plywood structure. 

In  the following calculations we  shall use the estimate, proven in \cite{Briane1}, 
\begin{lemma}[\cite{Briane1}]\label{Differ_Charact}
For the characteristic function of  a fibre system yields
\begin{equation*}
 ||\vartheta_r(x+\tau) - \vartheta_r(x)||^2_{L^2(\Omega)} \leq  C r L |\tau|,
\end{equation*}
where $L$ is the length  and $r$ is the radius of fibres.
\end{lemma}
 
 Since in each $\Omega_n^\ve$ the length of  fibres is of order $\ve^r$, applying  Lemma~\ref{Differ_Charact}, equality \eqref{Taylor_Approx}, and  the estimates $N_\ve\leq C \ve^{-3r}$ and $I_n^\ve\leq C \ve^{3(r-1)}$  we conclude that
 \begin{equation}\label{estim_non_local}
\sum\limits_{n=1}^{N_\ve} \int\limits_{\Omega_n^\ve}\sum\limits_{j=1}^{I_n^\ve} \left| \vartheta_\ve\big(R_{k_j^{n}}(x-x_{k_j^n}^\ve) \big)- 
\vartheta_\ve\big(R_{\kappa_n} (x-x_n^\ve) - W_{x_{n}^\ve}m_j\ve \big)\right|^2 dx\leq  C\ve^{3r-2}.  
\end{equation}
Considering the definition of  $E^\ve$ and $E^\ve_n$,   estimate \eqref{estim_non_local}  and the fact that $$
\chi_{\Omega_f^\ve}(x)=\chi_\Omega(x) \sum_{n=1}^{N_\ve} \sum_{j=1}^{I_n^\ve} \vartheta_\ve\big(R_{k_j^{n}}(x-x_{k_j^n}^\ve)\big)
$$
for $N_\ve$, $I_n^\ve$ and $k_j^n$ as defined above,  we obtain  
\begin{eqnarray}\label{estimnon-perLoc}
\int_{\Omega} |(E^\ve(x) - E^\ve_n(x)) e(u^\ve) e (\phi)| dx\leq C \ve^{3r-2} \|u^\ve\|_{H^1(\Omega)} \|\phi \|_{W^{1,\infty}(\Omega)}
 \end{eqnarray}
and  for  $2/3 <r<1$  and $\{u^\ve\}$ bounded in $H^1$ converges to zero as $\ve \to 0$. Thus  in the definition of a locally-periodic approximation  we shall consider a covering of $\Omega$ by cubes of side $\ve^r$ with $2/3 <r<1$. 

Now we take $ \psi(x)= \psi_1(x) + \ve(\mathcal L^\ve_\rho\psi_2)(x) $
as a  test function in \eqref{Aprox_Eq}, where  $\psi_1\in C^\infty_0(\Omega)$ and  $\psi_2\in W^{1,\infty}_0(\Omega; C^\infty_{\text{per}}(Y_x))$.
Applying to  the first integral in \eqref{Aprox_Eq}  similar calculations as for the locally-periodic problem \eqref{main2},  using   \eqref{estimnon-perLoc}  and the locally-periodic two-scale convergence of a subsequence of $\{\nabla u^\ve\}$  we obtain 
\begin{eqnarray*}
 &&\int_\Omega \ddashinttt_{Y_x} B(x, y) \big[e(u) + e_y(u_1(x,y))\big] \big[e(\psi_1)+ 
e_y\big(\psi_2(x, y)\big) \big] dy dx=\int_\Omega  G(x)\, \psi_1\, dx,
\end{eqnarray*}
where $Y_x=D_xY$.  The  transformation $\mathcal F: Y_x \to Z_x$, i.e. $\tilde y= \mathcal F(y)=R_{x_3}y$, gives
\begin{eqnarray*}
&&\int_\Omega \,  \ddashinttt_{Z_x}   \tilde B(x,\tilde y) \big( e(u) + e^R_{\tilde y}(\tilde u_1)
\big)\big(e(\psi_1)+ e^R_{\tilde y}(\tilde \psi_2)  \big)d\tilde y dx =\int_\Omega  G(x)\, \psi_1\, dx,
\end{eqnarray*}
with  
$\tilde B(x,\tilde y)= E_1\hat \vartheta(x,\tilde y) +E_2(1- \hat \vartheta(x, \tilde y))$, and $e^R_{\tilde y}$ as in \eqref{def_trans_grad}.

We notice that $\hat\vartheta$  is independent of $\tilde y_1$  and, similarly as in the locally-periodic situation,  
we can conclude   that the correspondent unit cell problems are two-dimensional.
\begin{theorem}\label{Theorem_NonPer_macro}
 The sequence of  solutions $\{u^\ve\}$ of microscopic model  \eqref{main} with the non-periodic elasticity tensor $E^\ve$, determined by the characteristic function \eqref{elastnonper}, converges to a  solution of the macroscopic problem 
\begin{eqnarray*}
- \text{div } (B^{\text{hom}}(x) e(u)) &=& G \quad \text{ in } \Omega, \\
u &=& g \quad \text{ on } \partial\Omega,
\end{eqnarray*}
where the  homogenized elasticity tensor is given by 
\begin{eqnarray*}
B^{\text{hom}}_{ijkl}(x)
&=&\ddashinttt_{\hat Z_x} \left( \tilde B_{ijkl}(x,\hat y) +\tilde  B(x,\hat y) \hat e^R_{\hat y}(\tilde \omega_{ij})_{kl} \right) d\hat y 
\end{eqnarray*}
and $\tilde \omega_{ij}$ are solutions of the cell problems 
\begin{eqnarray}\label{UnitCellNonPer}
 -\nabla_{\hat y}\cdot \Big(
  \hat R_{x_3} \tilde B(x,\hat y) \hat e^R_{\hat y} (\tilde \omega_{ij}) \Big) 
 =  \nabla_{\hat y}\cdot \Big(\hat R_{x_3}  \tilde B(x,\hat y) l_{ij} \Big) && \quad \text{ in } \hat Z_x,\nonumber\\
\tilde{\omega}_{ij} \hspace{1.0 cm} \text{ periodic } && \quad  \text{ in } \hat Z_x. \nonumber
\end{eqnarray}
Here $\hat y=(\tilde y_2, \tilde y_3)$,  $l_{ij}= \frac 12 (l_i \otimes l_j + l_j \otimes l_i)$,  where $(l_i)_{1\leq i\leq 3}$ is the canonical basis of $\mathbb R^{3}$, the matrix $\hat R_{x_3}$ and  $\hat e^R_{\hat y}$ are given by \eqref{transform_reduce}, and  $\hat Z_x= Z_x \cap \{\tilde y_1=0\}$.
\end{theorem}

The same arguments as for locally-periodic problem imply the existence of a unique solution of the macroscopic problem and the convergence of the entire sequence of solutions of the microscopic models. 

We notice that for non-periodic plywood-structure, where $r=1$,  the  space-dependent periodicity and the unit-cell problem  \eqref{UnitCellNonPer}   differ from those obtained for the locally-periodic microstructure with $0<r<1$, compare to \eqref{unit_loc_per2}. We  also emphasise  that in the case of locally-periodic microstructure  the form of the  macroscopic model is the same for every $r \in(0,1)$, see Theorem~\ref{Theorem_macro_loc_per}.

\section{Conclusion}
In this paper we investigate the  concept of the locally-periodic two-scale convergence.  Similar to the periodic case,  we use  the idea of oscillating test functions, which are synchronous with oscillations in either the microstructures or in coefficients of microscopic problems. However, we extend the theory to the non-periodic case, in particular we focus on locally-periodic structures.  We derived the macroscopic equations for a linear elasticity problem, posed in a domain with a ``plywood structure", a prototypical pattern in many biomaterials such as bones or exoskeletons. The  non-periodic microstructure can be  approximated by a locally-periodic one, provided the transformation matrix is twice continuous differentiable. The  techniques developed  here  are not restricted to the equations of linear elasticity and can be applied to a wide range of stationary or  time-dependent problems. For example,  a heat conduction problem was considered in \cite{Briane3} and a macroscopic equation was derived using the $H-$convergence method.
 Our  results would lead to the same macroscopic equation and it appears that  the derivation would follow in a much more direct manner. Moreover,  our approach allows multiscale analysis in domains with more general microscopic geometries than those considered in \cite{Alexandre, Mascarenhas}.
In the context of the definition of  a microstructure given by the transformation of centres of spherical balls,  see Fig~\ref{fig51}, considered in \cite{Alexandre, Briane3}, we have the relation  $\theta^{-1}(x)= D^{-1}(x)x$, where  $\theta$ is the $C^2$-diffeomorphism defining the transformation of centres of  balls.

\begin{figure}
\begin{center}
\includegraphics[width=0.43\textwidth]{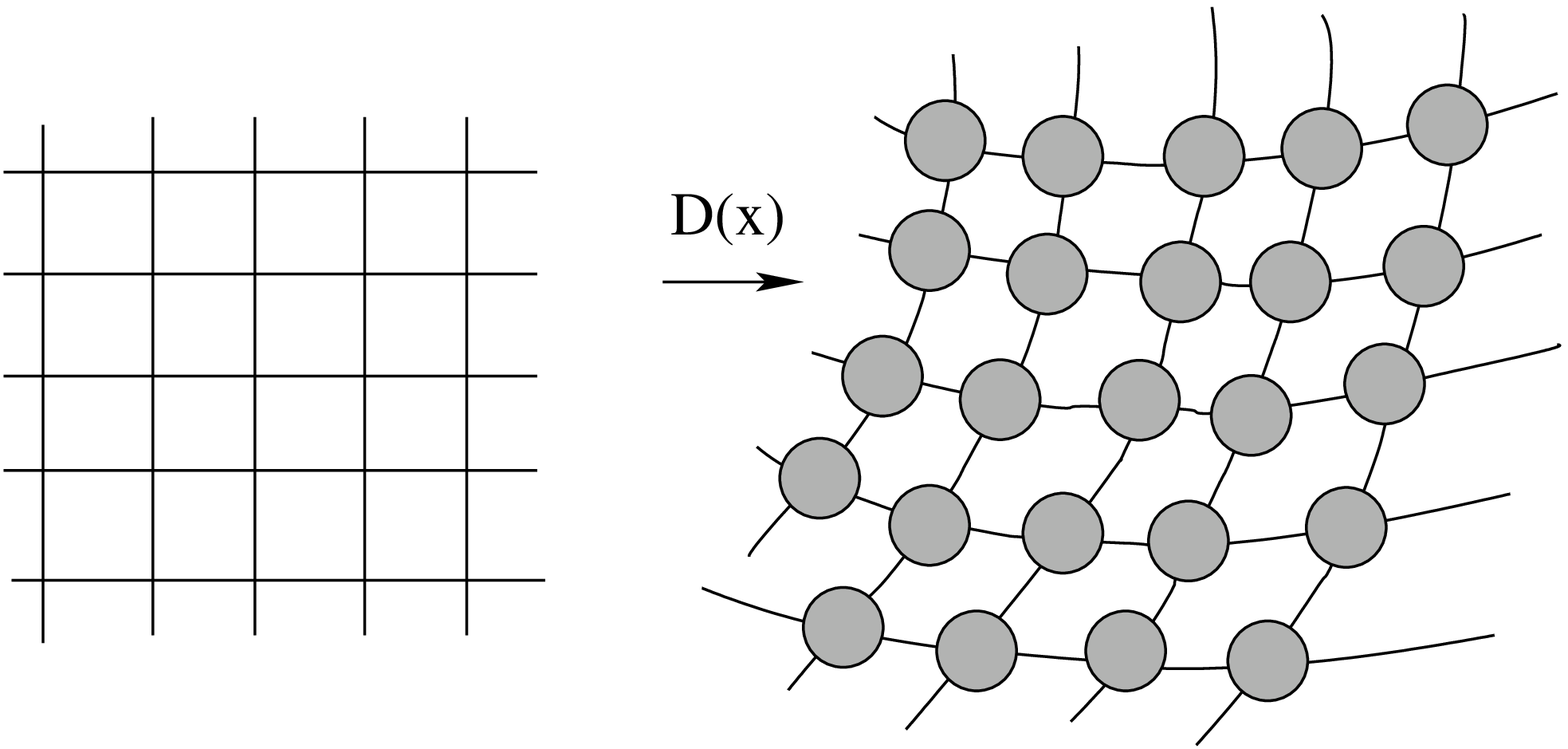}\qquad 
\includegraphics[width=0.21\textwidth]{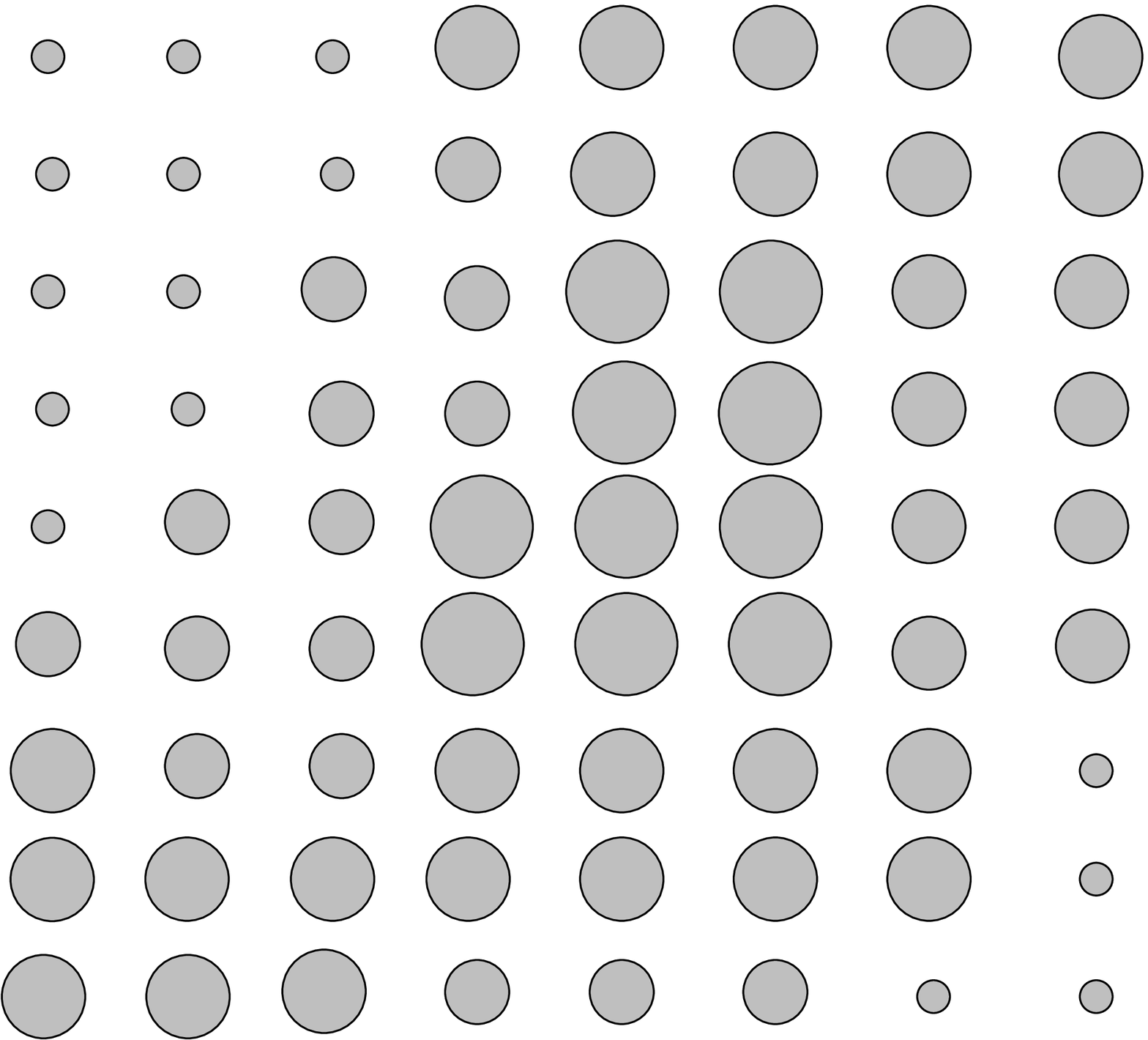}
\caption{\small{Transformation of centres of spherical balls. Space-dependent perforations.}} \label{fig51}
\end{center}
\end{figure}
Another example of a locally-periodic microstructure  is the space-dependent perforation in concrete materials, see Fig~\ref{fig51}, where  a heterogeneity of the medium  is given by  areas of high and low diffusivity, \cite{AdrianTycho2}.  For  $\rho\in C^1(\mathbb R^3)$, such that $0<\rho_1\leq \rho(x) \leq \rho_2 <1$ for $x\in \mathbb R^3$, we consider    $Y$-periodic function
\[
 \chi(x,y) = \begin{cases}
               1 \quad  & \text{ for } \quad |y|\leq \rho(x), \\
0\quad &\text{ for } \quad  |y|> \rho(x).
              \end{cases}
\]
  Then  the characteristic function of a domain with low diffusion is given by $\chi_{\Omega^\ve_l}(x)=(\mathcal L^\ve_0\chi)(x)$.
In the notation of \cite{AdrianTycho2},  the corresponding level set function reads $S(x,y)=|y|^2-\rho^2(x)$.
Showing that the locally-periodic problem provides a correct approximation for the non-periodic model and applying 
the locally-periodic two-scale convergence  with $D(x)=I$, we should obtain the same macroscopic equations, as derived in \cite{AdrianTycho2} using formal asymptotic expansion.  The main step of the  approximation  involves the following calculations.
For $\ve\kappa_n\in \Omega_n^\ve$ and $k_j^n= \kappa_n + m_j$, with $j=1, \ldots I_n^\ve$ and $m_j \in \mathbb Z^3$,  considering Taylor expansion for $\rho(x)$ around $\ve\kappa_n$, we have 
 \begin{eqnarray*}
&& \sum\limits_{n=1}^{N_\ve}\int\limits_{\Omega_n^\ve} \sum\limits_{j=1}^{I_n^\ve}
\Big| \chi\big(\ve\kappa_n, \frac x \ve\big) - \chi\big(\ve k_j^n, \frac x \ve\big) \Big|^2 dx  
\leq   \sum\limits_{n=1}^{N_\ve}I_n^\ve
 \sup\limits_{1\leq j \leq I_n^\ve}\Big\| \chi\big(\ve\kappa_n, \frac x \ve\big) - \chi\big(\ve k_j^n, \frac x \ve\big) \Big\|^2_{L^2}\leq \\
&& c_1 \sum\limits_{n=1}^{N_\ve} I_n^\ve \sup\limits_{1\leq j \leq I_n^\ve}\Big| \big|\ve \rho(\ve \kappa_n)\big|^3 - \big|\ve \rho(\ve \kappa_n) \pm \ve 
 \|\nabla \rho \|_{L^\infty}|\ve m_j|\big|^3 \Big| \leq c_2  \sup\limits_{1\leq j \leq I_n^\ve}|\ve m_j| \leq c \ve^r.
\end{eqnarray*}

\section*{Acknowledgments}
The author  would like to thank Christof Melcher, Yuriy Golovaty and  Fordyce Davidson  for fruitful discussions. 
Special thanks go  to the anonymous reviewers for their suggestions and comments which have significantly improved the presentation of this paper.

\end{document}